\documentclass[reqno,final]{amsart}%%%%%%%%%%%%%%%%%%%%%%%%%%%%%%%%%%%%%%%%%%%%%%%%%%%%%%%%%%%%%%%%%%%%%%%%%%%%%%%%%%%%%%%%%%%%%%%%%%%%%%%%%%%%%%%%%%%%%%%%%%%%%%%%%%%%%%%%%%%%%%%%%%%%%%%%%%%%%%%%%%%%%%%%%%%%%%%%%%%%%%%%%%%%%%%%%%%%%%%%%%%%%%%%%%%%%%%%%%%%%%%%%%%%%%%%%%%%%%%%%%%%%%%%%%%% 
\usepackage{amssymb,amsmath,amsthm}
\usepackage{amscd,cite}
\usepackage{amsfonts}
\usepackage{amssymb}
\usepackage{latexsym}
\usepackage{color}
\usepackage{esint}
\usepackage{graphicx}
\usepackage{caption}
\usepackage{subcaption}
\usepackage{amsthm}
\usepackage{verbatim}
\usepackage{showkeys}
\usepackage{hyperref}
\usepackage{todonotes}
\usepackage{ulem}

\usepackage[makeroom]{cancel}
\usepackage{mathtools}

\DeclareMathOperator*{\esssup}{ess\,sup}

\let\hide\iffalse

\newtheorem{theorem}{Theorem}[section]

\newtheorem{corollary}[theorem]{Corollary}

\newtheorem{lemma}[theorem]{Lemma}

\newtheorem{remark}[theorem]{Remark}

\renewcommand{\S}{\mathbb{S}}

\newcommand{\be}{\begin{equation}}
\newcommand{\bm}{\begin{multline}}
\newcommand{\ee}{\end{equation}}

\newcommand{\Bes}{\begin{eqnarray*}}
	\newcommand{\Ees}{\end{eqnarray*}}
\newcommand{\Be}{\begin{equation}}
\newcommand{\Ee}{\end{equation}}

\def\B{\begin{equation}}
\def\E{\end{equation}}
\def\BN{\begin{eqnarray*}}
\def\EN{\end{eqnarray*}}

\numberwithin{equation}{section}

\usepackage{tikz-cd}
\usepackage{ragged2e}
\usepackage{etoolbox}
\usepackage{lipsum}

\usepackage{color}

	\title[Cooling process and upper-bounds for the Boltzmann equation]{%Cooling Process and Pointwise Upper Bounds for the Inelastic Boltzmann Equation 
 Quantitative pointwise estimates of the cooling process for inelastic Boltzmann equation}
	\author{Gayoung An}
	\address{Department of Mathematics, Pohang University of Science and Technology (POSTECH), South Korea}
	\email{agy19@postech.ac.kr}
	
	\author{Jin Woo Jang}
	\address{Department of Mathematics, Pohang University of Science and Technology (POSTECH), South Korea}
	\email{jangjw@postech.ac.kr}
 \author{Donghyun Lee}
	\address{Department of Mathematics, Pohang University of Science and Technology (POSTECH), South Korea}
	\email{donglee@postech.ac.kr}

\begin{document}
\begin{abstract}
In this paper, we study the homogeneous inelastic Boltzmann equation for hard spheres. We first prove that the solution $f(t,v)$ is bounded pointwise from above by \(C_{f_0}\langle t \rangle^3\) and establish that the cooling time is infinite (\( T_c = +\infty \)) under the condition \( f_0 \in L^1_2 \cap L^{\infty}_{s} \) for \( s > 2 \). Away from zero velocity, we further prove that $ f(t,v)\leq
    C_{f_0, |v|} \langle t \rangle $ 
         for \(v \neq 0\) at any time \( t > 0 \). This time-dependent pointwise upper bound is natural in the cooling process, as we expect the density near \( v = 0 \) to grow rapidly. 
 We also establish an upper bound that depends on the coefficient of normal restitution constant, $\alpha \in (0,1]$. This upper bound becomes constant when $\alpha = 1$, restoring the known upper bound for elastic collisions \cite{L1983}. Consequently, through these results, we obtain Maxwellian upper bounds on the solutions at each time. 
\end{abstract}

\maketitle
\tableofcontents

\section{Introduction}
The classical Boltzmann equation describes the behavior of a rarefied collisional gas, consisting of a very large number of identical
particles undergoing perfectly elastic binary collisions. The velocity distribution function $f(t,x,v)$ is the probable number density of particles at time $t$, at point $x$, having speed $v$. The Boltzmann equation is written as
\begin{align*}
    \partial_t f + v \cdot \nabla_x f = Q(f,f), \quad f(0,x,v)=f_0(x,v) 
\end{align*} for $x\in \Omega\subset \mathbb{R}^3$, $v \in \mathbb{R}^3$ and $t \geq0$.
On the left-hand side of the Boltzmann equation we see a transport operator which describes the translation of a particle distribution without the presence of other forces and collisions. The nonlinear quadratic term, $Q(f,f),$ on the right-hand side is added to account for the collisions between particles. When two particles, initially traveling at velocities \(v\) and \(v_*\), collide with each other, their velocities change due to the interaction. After the collision, the new velocities of the particles are denoted as \(v'\) and \(v_*'\), respectively. These post-collisional velocities \(v'\) and \(v_*'\) are determined by the specific dynamics of the collision process, which can depend on factors such as the masses of the particles and the nature of the interaction between them. We assume that the mass of each particle is identical and normalize it to be 1 throughout the paper. Under elastic collisions, according to the laws of conservation of momentum and energy, the following relationships hold:
\begin{equation}
    \label{conservation}
v + v_* = v' + v_*', \quad |v|^2 + |v_*|^2 = |v'|^2 + |v_*'|^2.
\end{equation}
From these conservation laws, we derive four constraints on the six variables \(v'\) and \(v_*'\), leaving two degrees of freedom. Specifically, the post-collisional velocities \(v'\) and \(v_*'\) can be expressed as:
\begin{equation}
    \label{n_elastic}
v' = v - ((v - v_*) \cdot n)n \quad \text{and} \quad v_*' = v_* + ((v - v_*) \cdot n)n,
\end{equation}
using the vector \(n \in \mathbb{S}_{+}^{2}\). Physically, if gas particles are represented as billiard balls, the vector \(n \in \mathbb{S}_{+}^{2}\), known as the impact direction, denotes the unit vector connecting the centers of the balls at the moment of collision, as illustrated in Figure \(\ref{defi_n_fig}\).

In this paper, we frequently use an alternative representation of the post-collisional velocities \(v'\) and \(v_*'\) involving another angular variable \(\sigma \in \mathbb{S}^{2}\). This representation is given by:
\[
v' = \frac{v + v_*}{2} + \frac{|v - v_*|}{2}\sigma \quad \text{and} \quad v_*' = \frac{v + v_*}{2} - \frac{|v - v_*|}{2}\sigma.
\]
In Figure \(\ref{fig_v'_elastic}\), we illustrate the possible locations of \(v'\) and \(v_*'\) for given \(v\) and \(v_*\). As shown in the figure, the polar angle of \(\sigma\) along the \(z\)-axis, corresponding to \(v - v_*\), varies from \(0\) to \(\pi\), while the polar angle of \(n\) ranges from \(0\) to \(\frac{\pi}{2}\), placing \(n\) in \(\mathbb{S}^{2}_{+}\). Interestingly, for fixed \(v\) and \(v_*\), \(v'\) and \(v_*'\) always lie on a single sphere, and the Jacobian determinant between \((v, v_*)\) and \((v', v_*')\) is always equal to 1.
\begin{figure}[t]
\centering
\includegraphics[width=4cm]{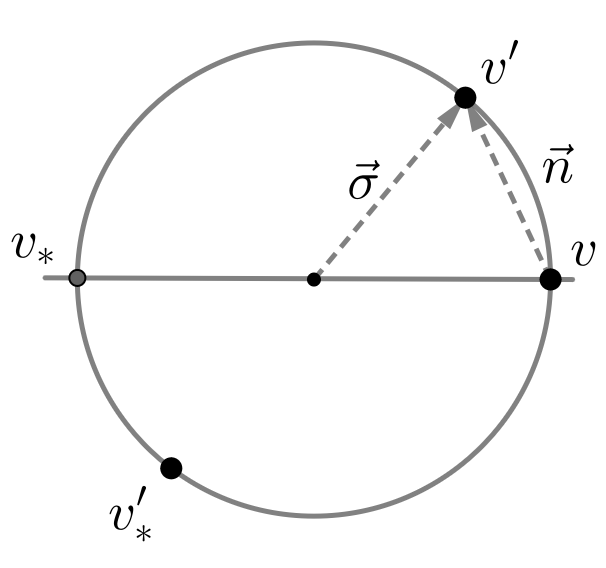} 
\caption{Given $v, v_*$, possible locations of $v'$, $v_*'$ in elastic collisions. \label{fig_v'_elastic}} 
\end{figure}
Next, we define the collision operator using the variables \(v\), \(v_*\), \(v'\), \(v_*'\), and \(n\) as described in \eqref{n_elastic}. It is defined as follows:
\begin{align} \label{classic BE}
\begin{split}
      Q(f,f)(t,v)
     &=\int_{\mathbb{R}^{3}} \int_{\mathbb{S}_{+}^{2}} B(|v-v_*|,\cos \theta) \left(f(t,v')f(t,v'_*)-f(t,v)f(t,v_*)\right) \; dn dv_*,  \\
     &:= Q^{+}(f,f)(t,v)-f(t,v) L f(t,v),
\end{split}
\end{align} where
\begin{align*}
     \cos \theta = \left( \frac{v_*-v}{|v_*-v|}, n \right), 
     \quad \theta \in \left[0,\frac{\pi}{2} \right].
\end{align*} Here, the operators are local in \(t\) and \(x\); we abuse notation by writing \(f(t,x,v) =f(t,v)\). The gain term, \(Q^{+}(f,f)(t,v)\), represents an increase in the density of particles with velocity \(v\) due to collisions involving velocities \(v'\) and \(v_*'\). Because the equations in \eqref{conservation} are time-reversible, if particles with velocities \(v\) and \(v_*\) collide to become \(v'\) and \(v_*'\), then particles with velocities \(v'\) and \(v_*'\) can collide to revert to \(v\) and \(v_*\).
The loss term, \(f(t,v)Lf(t,v)\), indicates a decrease in the density of particles with velocity \(v\) due to collisions with particles of velocity \(v_*\), resulting in their disappearance.

Now, we explain the collision kernel. Depending on the assumptions regarding the interaction potentials (e.g., hard-sphere potential or inverse-power-law potential), the collision kernel can be expressed as:
\[
B(|v - v_*|, \cos \theta) = b(\cos \theta) |v - v_*|^{\gamma}
\]
for some constant \(\gamma\). For instance, when particles collide like billiard balls (hard-sphere potential), we write \(B(|v - v_*|, \cos \theta) = |(v - v_*) \cdot n|\).
If \(\int_{\mathbb{S}_{+}^{2}} b(\cos \theta) \, dn < \infty\) or if \(b(\cos \theta) \in L^\infty(\mathbb{S}_{+}^{2})\), the collision kernel is called a cut-off kernel (or is with angular cutoff), following Grad's angular cutoff criterion in \cite{G1963}. Otherwise, it is referred to as a non-cutoff kernel (or is without angular cutoff). In the case of a non-cutoff collision kernel, there is a nonintegrable singularity around zero for \(\theta\). (Some authors place the singularity near \(\pi/2 - \theta \approx 0\).) Recently, the vanishing angular singularity limit from the inverse-power-law potential to the hard-sphere potential has been studied in \cite{MR4565088}.

By the famous Boltzmann H-theorem, we formally see that the total entropy of the system $-\iint f\log f dxdv$ will monotonically increase in time. Then we expect that the distribution will converge (in some sense) to an entropy maximizing state. One of the candidates for the states of maximum entropy is so-called the (local) Maxwellian distribution which is given by \begin{align*}
     \mathcal{M}(t,x,v) = \rho(t,x) \frac{e^{-\frac{|v-u(t,x)|^2}{2T(t,x)}}}{(2\pi T(t,x))^{3/2}},
 \end{align*}
 where the Boltzmann constant $k_B$ is assumed to be 1. Here, $\rho,$ $u$, and $T$ correspond to the physical terms such as mass density, macroscopic fluid velocity, and local temperature, respectively. A simple calculation can show that $Q(\mathcal{M}, \mathcal{M})(t,x,v)=~0$. If $\rho$, $u$, and $T$ are chosen so that the left-hand side of the Boltzmann equation satisfies $(\partial_t + v \cdot \nabla_x)\mathcal{M}=0$, $(\rho,u,T)$ solves some macroscopic conservation laws. In this case, the local Maxwellian $\mathcal{M}$ can be written in some specific forms depending on the domain and boundary conditions. We refer to \cite{MR1031086} for specific forms. In the paper, Desvillettes actually
%implying that $\rho$, $u$, and $T$ solve some macroscopic conservation laws, then the local Maxwellian $\mathcal{M}$ is indeed a steady-state.
%One example of this is when the local physical terms $\rho$, $u$, and $T$ are constant. In this scenario, $\mathcal{M}$ no longer depends on $t$ or $x$, and is referred to as the global Maxwellian distribution. Desvillettes, in \cite{MR1031086},
proved that the renormalized solution of the Boltzmann equation converges weakly in $L^1$ to a local time-dependent Maxwellian state in the large-time asymptotics within a bounded domain. It was shown that this local Maxwellian further satisfies $(\partial_t + v \cdot \nabla_x)\mathcal{M}=0$. 
Additionally, before this result, Arkeryd proved in \cite{arkeryd1988long} the convergence of a renormalized solution to the global Maxwellian. The local stability of the global Maxwellian has now been very well-known. See \cite{MR2214953, AMUXY_1, AMUXY_2, Strain} for the proof of local stability nearby equilibrium for both cutoff and noncutoff cases. The existence of global solutions was firstly proved by Carleman \cite{MR1555365} in 1933 for the spatially homogeneous case. Then for the spatially inhomogeneous case, the global existence was proved  in the perturbative framework \cite{MR0363332} and near vacuum \cite{MR0760333}. Then the global existence of renormalized solutions to the Boltzmann equation by DiPerna-Lions \cite{MR1014927} in 1989 which does not require size restriction on the initial data. This framework has been extended to the noncutoff case by Alexandre and Villani in \cite{MR1857879}. The nonlinear energy method was developed to prove the global existence to the global wellposedness of the Vlasov(-Poisson or -Maxwell)-Boltzmann system in \cite{Guo_VPB,Guo_VMB, MR2259206}. See also \cite{MR2043729} for the energy method for the Boltzmann equation. 
In \cite{Guo_VPL,MR2209761}, the local stability of the Maxwellian for the Landau equation was proved based on the construction of the classical solution via the nonlinear energy method developed in \cite{MR1946444}.  Regarding the global existence and stability in a bounded domain with general physical boundary conditions (e.g., inflow, bounce-back, specular, and diffuse reflection), see \cite{Guo_decay,CD2018, MR3562318}.  
 \begin{figure}[t]
\centering
\includegraphics[width=4cm]{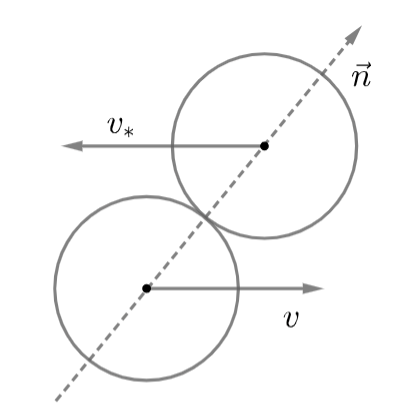} 
\caption{Impact direction, $\vec{n}$. \label{defi_n_fig}} 
\end{figure}

In contrast to an elastic collision, an inelastic collision is a collision in which kinetic energy is not conserved, and this type of gas is referred to as granular gas. Physically, in each inelastic collision, there is some loss of momentum in the impact direction, though the total momentum remains conserved. Let \(0 < \alpha \leq 1\) be the coefficient of normal restitution. For the impact direction \(n \in \mathbb{S}_{+}^2\), the following equations hold:
\begin{align} \label{alpha}
\begin{split}
    (v' - v_*') \cdot n &= -\alpha (v - v_*) \cdot n, \\
    (v' - v_*') - ((v' - v_*') \cdot n)n &= (v - v_*) - ((v - v_*) \cdot n)n.
\end{split}
\end{align}
If \(\alpha = 1\), the collision is elastic, satisfying \eqref{conservation}. If \(\alpha = 0\), the collision is sticky, meaning the particles travel together after the collision. Generally, the coefficient \(\alpha\) depends on \(|v - v_*|\) and \(\theta\), the angle between \(v - v_*\) and \(n\). However, in many cases, this dependence is not considered significant, and a constant restitution coefficient is often used.
Further details about inelastic collisions and the Boltzmann equation will be covered in Section \ref{sec.prelim}, Preliminaries. 

Many physicists and mathematicians are interested in the behavior of solutions to both the elastic and inelastic Boltzmann equations. Numerous numerical, physical, and mathematical studies have been conducted on this topic. In mathematics, the \(L^{\infty}\) norm, weighted \(L^{\infty}\) norm, \(L^{1}\) norm, and weighted \(L^{1}\) norm of the Boltzmann equation were among the first to be studied.

The \(L^{\infty}\) norm is particularly important because it allows us to consider the pointwise bounded behavior of the solution, which is much stronger than merely being energy-bounded (i.e., \(\int f|v|^2 \, dv < +\infty\)). Based on \(L^{\infty}\) estimates, the upper Maxwellian bounds were established in \cite{GPV2009}, which are used in the local existence theory of strong solutions as seen in \cite{G1958, KS1978}. Additionally, the Hölder continuity modulus and the lower Maxwellian bounds are dependent on the \(L^{\infty}\) norm of the solution, as discussed in \cite{IMS2020_1, IS2019}. This lower bound is crucial for demonstrating the convergence of solutions to Maxwellian equilibrium, as shown in \cite{DV2005}.
One of the most significant properties of inelastic collisions is that they are not time-reversible, making it impossible to determine the direction of entropy with respect to time. In the case of elastic collisions, the maximum entropy principle and Boltzmann’s \(H\)-theorem ensure that the solution converges to the Maxwellian distribution. However, for inelastic collisions, the absence of the \(H\)-theorem makes it more challenging to establish an equilibrium state.

Moreover, the behavior of solutions varies depending on the additional physical forces, such as a heat bath and friction, that are incorporated into the kinetic model. Without external forces, the kinetic energy decreases until all gas particles stop, a state referred to as the \textit{cooling state}. Details about these models are introduced in \cite{V2006} by Villani. Notably, in \cite{SCM2006}, Mischler, Mouhot, and Ricard studied the \textit{cooling process} and proved that the solutions converge to a Dirac mass in the \mbox{weak-*} measure sense asymptotically over a long time.

If kinetic energy is lost but the total density is conserved, it is predicted that the number of particles with low velocities will increase over time, while the number of particles with high velocities will decrease. Moreover, as velocity approaches zero, the density function increases rapidly over time (see Figure \ref{graph}). This motivates our study of the behavior of solutions over time for inelastic collisions without any external forces in this paper.

In this paper, we study the \(L^{\infty}\) norm of the solution of the homogeneous inelastic Boltzmann equation for hard spheres. We will then extend the \(L^{\infty}\) norm to establish upper Maxwellian bounds. During the cooling process, as the solution tends towards \(\delta_{v=0}\) and the kinetic energy approaches zero over time, it becomes crucial to derive pointwise upper bounds for the solution.
\begin{figure}[t]
\centering
\includegraphics[width=6cm]{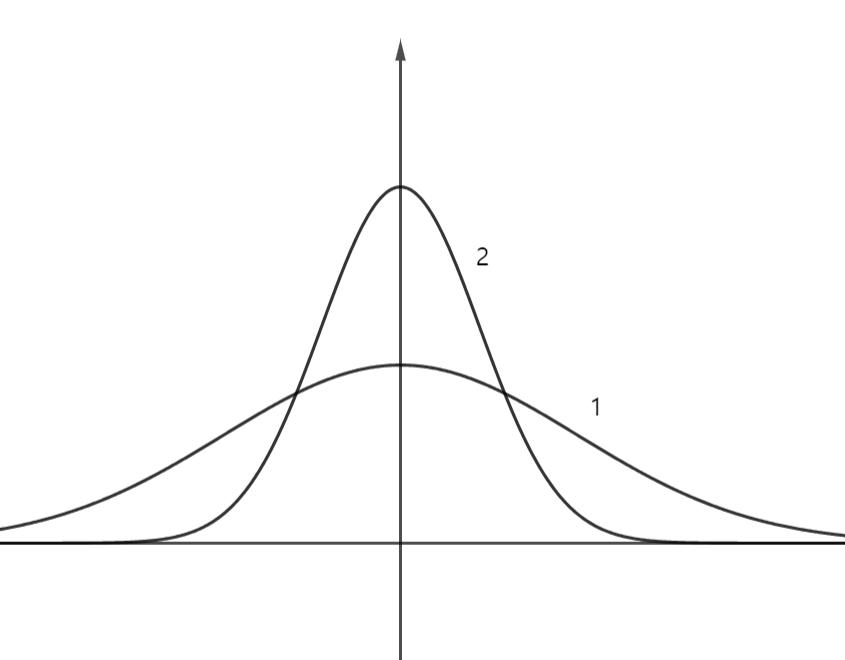} 
\caption{Velocity distribution for inelastic collisions over time (1 $\rightarrow$ 2).\label{graph}} 
\end{figure}

Before describing the inelastic model, we will briefly review the literature on the behavior of solutions to both the Boltzmann and inelastic Boltzmann equations.

\subsection{Notations}
Here, we briefly introduce some functions spaces, norms, and our notations. We define the weighted norms
\begin{align*}
    \|f\|_{p,s} &= \|f(t,\cdot)\|_{p,s}=\left(\int_{\mathbb{R}^3} \left(f(t,v)(1+|v|^2)^{\frac{s}{2}}\right)^p \; dv \right)^\frac{1}{p} < \infty
\end{align*} for $0 \leq s$ and $1 \leq p \leq \infty$. The corresponding weighted spaces are defined as 
\begin{align*}
    L_{s}^{p} = L_{s}^{p}(\mathbb{R}^3) = \left\{f :\text{ $f$ is measurable on} \; \mathbb{R}^3, \; \|f\|_{p,s}<\infty \right\}.
\end{align*} If $p=\infty, s=0$, we have 
\begin{align*} 
       \|f\|_{\infty,0} = \|f\|_{\infty}= \esssup_{v\in \mathbb{R}^3} |f(x)|.
\end{align*}  The indicator function of a subset $A$ within a set $X$ is a function $\chi_{A} : X \rightarrow \{0,1\}$, defined as
	\begin{align*} 
	\chi_{A}(x) = \begin{cases}
	1,\quad x \in A, \\ \;0,\;\quad x \notin A
	\end{cases} 
	\end{align*} for $x \in X$. Throughout the paper, we denote as $A\lesssim B$ to mean that there exists a uniform constant $C$ such that $A\le CB$. If the constants depend on $K$, we write $C(K)$.

\subsection{Known results for the elastic collisions}
In particular, if the velocity distribution function \(f\) does not depend on the spatial variable \(x\), the Boltzmann equation is called the (spatially) homogeneous Boltzmann equation. Regarding the homogeneous Boltzmann equation with a cut-off collision kernel, Carleman \cite{T1957} studied \(L^{\infty}\) estimates for hard spheres. In 1982, Arkeryd extended this result to the general case in \cite{L1983} and proved that
\[
(1+|v|)^{s'}f(t,v) \in L^{\infty}(\mathbb{R}_v^3)
\]
for \(t>0\) under the assumption that the initial distribution \((1+|v|)^{s}f_0 \in L^{\infty}(\mathbb{R}_v^3)\) for \(s>2\). For \(s \leq 5\), this holds for any \(s' \leq s\). For \(s > 5\), it holds that \(s' < s\) under the condition \((1+|v|^2)^{s_1/2}f \in L^{1}(\mathbb{R}^3_v)\) for some \(s_1 > 2\). In 1997, Bobylev proved that there exists \(\theta_*\) such that \(0 < \theta_* \leq \theta\) and
\[
e^{\theta_*|v|^2}f \in L^1(\mathbb{R}_v^3)
\]
for \(t > 0\) whenever \(e^{\theta|v|^2}f_0 \in L^1(\mathbb{R}_v^3)\) for hard spheres in \cite{B1997}. In \cite{P1997}, Pulvirenti and Wennberg demonstrated that the solution is pointwise bounded from below by a Maxwellian for hard potentials. In 2009, Gamba, Panferov, and Villani proved that if \(f_0\) has a Maxwellian upper bound, then the solutions have a uniform Maxwellian upper bound for hard potentials in \cite{GPV2009}. They applied the comparison principle using a dissipative property of the linear Boltzmann equation. Furthermore, the propagation of upper Maxwellian bounds in the spatially inhomogeneous case, specifically within the unit hypercube with periodic boundary conditions, was discussed.

Regarding the spatially inhomogeneous case with a non-cutoff collision kernel, Silvestre \cite[Theorem 1.2]{S2016} proved that
\[
\|f(t,\cdot,\cdot)\|_{L^{\infty}_{x,v}} \leq C(t) < \infty,
\]
where \(C(t) < \infty\) for all \(t > 0\) and depends only on the bounds of local macroscopic mass, energy, and entropy densities. In \cite[Theorem 1.3-(1)]{IMS2020}, Imbert, Mouhot, and Silvestre obtained
\[
\sup_{x,v} f(1+|v|)^{q} < \infty
\]
for \(t > 0\) if \(\sup_{x,v} f_0 (1+|v|)^{q} < \infty\) for \(q \geq 0\), applicable to hard and moderately soft potentials. In \cite{IMS2020_1}, the authors demonstrated the appearance of a Gaussian lower bound. Briant proved the emergence of an exponential lower bound with the physical Maxwellian diffuse and specular reflection boundary conditions in \cite{B2015, B2015_2}. In \cite{F2021}, Fournier established the creation and propagation of exponential moments for the homogeneous Boltzmann equation for hard potentials. However, the propagation and creation of a Maxwellian upper bound remain open in the non-cutoff case.

\subsection{Known results for the inelastic collisions}
When gas particles collide with each other and lose energy, it is termed an inelastic collision, and such a gas is referred to as a granular gas. The behavior of a granular gas differs significantly from that of an ideal gas undergoing elastic collisions. The basic concepts of kinetic theory and exciting phenomena in granular gas dynamics were introduced in \cite{MR2101911}. A hydrodynamic description for near elastic particles was explained in \cite{MR1749231}. Here, we review various studies on the homogeneous inelastic Boltzmann equation for hard spheres.

In 2004, Bobylev, Gamba, and Panferov \cite{BGP2004} obtained the steady velocity distributions of the solutions for large \(|v|\) with external forcing. In \cite{GPV2004}, Gamba, Panferov, and Villani studied the existence, smoothness, uniqueness, and lower bound estimates of the solutions, incorporating an additional diffusion term in the equation. In \cite{SCM2006}, Mischler, Mouhot, and Ricard developed the Cauchy theory and demonstrated the creation and propagation of exponential moments. Mischler and Mouhot further proved the existence, uniqueness, and time asymptotic convergence of the self-similar solution in \cite{SC2006, SC2009}. They also proved Haff's law, which explains the precise rate of decay to zero for the granular temperature. Alonso and Lods also investigated a system of viscoelastic particles in \cite{MR3238524}. A cooling inelastic Maxwell molecules of a suitable metric in the set of probability measures was studied by Bisi, Carrillo and Toscani in \cite{MR2264621}. The Cauchy problem for small data in the space of functions bounded by Maxwellians was studied in \cite{R2009}. 

Recently, there have been further developments regarding the homogeneous inelastic Boltzmann equation with a non-cutoff collision kernel. An and Lee \cite{AL2023} proved the non-Maxwellian lower bound of the solutions. Qi \cite{MR4373214} studied the well-posedness theory of measure-valued solutions. Additionally, Jang and Qi \cite{JQ2023} established the global-in-time existence of measure-valued solutions and demonstrated the creation of polynomial moments.

The collapse phenomenon of inelastic particles has been discussed in \cite{DE1999, BCP1997}. Recently, \cite{DV2024, DV2024_1} studied a dynamical system describing this collapse. Additionally, the inelastic collision operator appears in more generalized physical settings, such as the scattering of photons interacting with both the ground-state gas and the excited-state gas in the radiative transfer process (see \cite{MR4382578, MR4618201}). Furthermore, \cite{MR3231051} examined the Cauchy problem for the inelastic Vlasov–Poisson–Boltzmann system with a soft potential in the near vacuum regime.

\subsection{Preliminaries} \label{sec.prelim}
\subsubsection{Aftermath of inelastic collisions}In this section, we explain inelastic collisions. Let \( v \) and \( v_* \) be the velocities of two particles before a collision, and let \( v' \) and \( v_*' \) be their velocities after the collision. Recall equation \eqref{alpha}:
\begin{align} \label{def_alpha_2}
\begin{split}
    (v' - v_*') \cdot n &= -\alpha (v - v_*) \cdot n,  \\
    (v' - v'_*) - ((v' - v_*') \cdot n) n &= (v - v_*) - ((v - v_*) \cdot n) n,
\end{split}
\end{align}
where \( \alpha \) is the coefficient of normal restitution, \( 0 < \alpha \leq 1 \), and the impact direction \( n \in \mathbb{S}^2_{+} \). We define
\begin{align}\label{def_beta}
\beta = \frac{1 + \alpha}{2}, \quad \frac{1}{2} < \beta \leq 1.
\end{align}
Since momentum is conserved, we set
\[
v' = v - a(v, v_*, n) n, \quad v_*' = v + a(v, v_*, n) n
\]
for some function \( a(v, v_*, n) \) and \( n \in \mathbb{S}^2_{+} \). We find \( a(v, v_*, n) = \beta ((v - v_*) \cdot n) n \), so that \( v' \) and \( v_*' \) satisfy equation \eqref{def_alpha_2}. The post-collisional velocities \( v' \) and \( v_*' \) are expressed as:
\begin{align} \label{v'}
       v' = v - \beta ((v - v_*) \cdot n) n, \quad v_*' = v_* + \beta ((v - v_*) \cdot n) n.
\end{align}
In Figure \ref{fig_v'}, we illustrate these coordinates. When \( v \) and \( v_* \) are fixed, there are two spheres: the possible locations of \( v' \) are restricted to the left sphere, while those of \( v_*' \) are in the right sphere. Another way to express \( v' \) and \( v_*' \) is by using \( \sigma \in \mathbb{S}^2 \), starting with \( \frac{v + v_*}{2} \) such that:
\begin{align}\label{v'_sigma}
\begin{split}
     v' &= \frac{v + v_*}{2} + \frac{1 - \beta}{2}(v - v_*) + \frac{\beta}{2} |v - v_*| \sigma,  \\
     v_*' &= \frac{v + v_*}{2} - \frac{1 - \beta}{2}(v - v_*) - \frac{\beta}{2} |v - v_*| \sigma.
\end{split}
\end{align}
\begin{figure}[t]
\centering
\includegraphics[width=6cm]{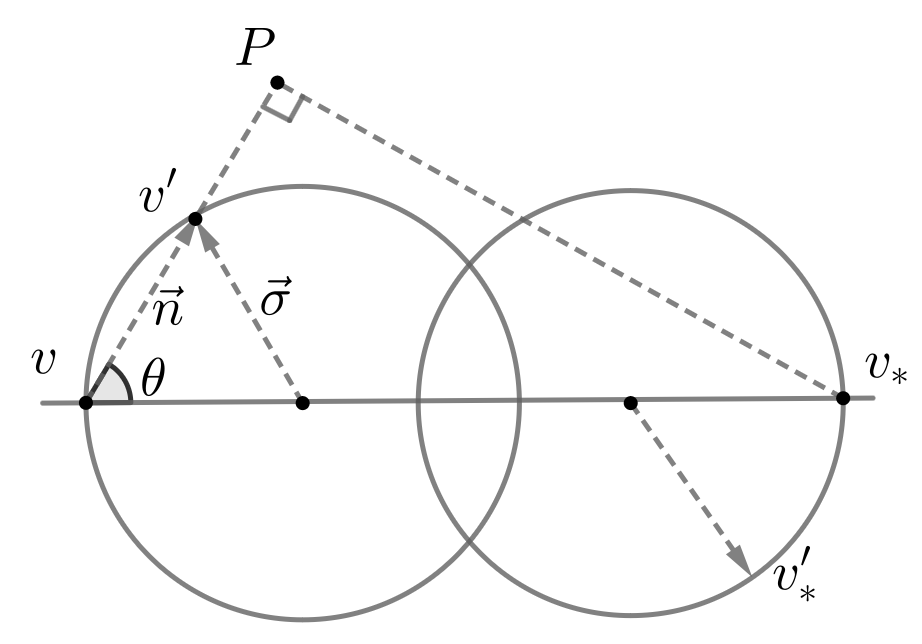} 
\caption{Possible locations of \( v' \) and \( v_*' \) given \( v \) and \( v_* \) in inelastic collisions. \label{fig_v'}} 
\end{figure}
In elastic collisions, we only need to define the post-collisional velocities \( v' \) and \( v_*' \). However, inelastic collisions are not time-reversible. If \( v' \) and \( v_*' \) collide, they do not revert to \( v \) and \( v_* \). Thus, in inelastic collisions, we define the pre-collisional velocities \( {'v} \) and \( {'v}_* \), which then become \( v \) and \( v_* \), respectively, after the collision. In equation \eqref{def_alpha_2}, we substitute \( v \) and \( v_* \) with \( v' \) and \( v_*' \), and \( {'v} \) and \( {'v}_* \) with \( v \) and \( v_* \) as follows:
\begin{align}
\begin{split} \label{'v}
    ({{'v}}-{{'v}}_*) \cdot n  &= -\frac{1}{\alpha} ( v-v_* ) \cdot n ,  \\
({'v}-{'v}_*)-(( {'v}-{'v}_*)\cdot n) n  &= (v-v_*)-((v-v_*) \cdot {n}) n.
\end{split}
\end{align} In contrast to \eqref{def_alpha_2}, $\alpha$ is replaced by $1/ \alpha$ in \eqref{'v}. Let $\gamma$ be
\begin{align} \label{gamma_defi}
    \gamma = \frac{1}{2} \left(1+\frac{1}{\alpha}\right), \quad 1 \leq \gamma < \infty.
\end{align}If we replace $\alpha$ with $1/ \alpha$ in \eqref{v'} and \eqref{v'_sigma}, then ${'v}, {'v}_*$ are expressed as
\begin{align} \label{'v,'v_*,n}
       {'v}= v-\gamma (( v-v_*) \cdot n ) {n}, \quad {'v}_*= v_*+\gamma (( v-v_*) \cdot n ) {n}
\end{align} for ${n}\in \S^{2}_{+}$ and 
\begin{align} \label{'v_sigma}
\begin{split}
     {'v}&=\frac{v+v_*}{2}+\frac{1-\gamma}{2}(v-v_*)+\frac{\gamma}{2}\;|v-v_*|{\sigma},  \\
     {'v}_*&=\frac{v+v_*}{2}-\frac{1-\gamma}{2}(v-v_*)-\frac{\gamma}{2}\;|v-v_*|{\sigma}
\end{split}
\end{align} for $\sigma \in \mathbb{S}^2$. The Jacobian between $(v,v_*)$ and $(v',v_*')$ is calculated as
\begin{align} \label{jacobian_alpha}
    \frac{\partial(v',v_*',-n)}{\partial(v,v_*,n) }=  \frac{\partial(v,v_*,-n)}{\partial({'v},{'v}_*,n) }=\alpha
\end{align} from \eqref{v'} and \eqref{'v,'v_*,n}. From \eqref{v'} and \eqref{v'_sigma}, we obtain directly
\begin{align} \label{loss_energy}
\begin{split}
      &|v'|^2+|v_*'|^2-|v|^2-|v_*|^2 \\
      &=- \frac{1-\alpha^2}{2}(( v-v_*) \cdot n )^2  \\
    &= - \frac{1-\alpha^2}{2}\frac{1-\frac{v-v_*}{|v-v_*|}\cdot \sigma }{2}|v-v_*|^2 \leq 0,
\end{split}
\end{align} and this shows that the loss of kinetic energy depends on the coefficient of normal restitution $\alpha$, $|v-v_*|$, and the collision angle. 

\subsubsection{Inelastic Boltzmann collision operator}
Now we introduce the spatially homogeneous inelastic Boltzmann equation. It goes by
\begin{align} \label{homo_IB}
    \partial_t f  = Q(f,f), \quad f(0,v)=f_0(v),
\end{align} 
where $v \in \mathbb{R}^3, \; t \geq0$ and $f\geq 0$. (The operator is local in $t$ and we abbreviate $t$ for notational convenience.) Taking $\phi(v)$ to be a suitably regular test function, the weak form of the inelastic Boltzmann collision operator $Q(f,g)(t,v)$ is written by 
\begin{align} 
     &\int_{\mathbb{R}^3} Q(f,g)(t,v) \phi(v) \; dv \notag \\
     &= \int_{\mathbb{R}^3}\int_{\mathbb{R}^3}\int_{\mathbb{S}_{+}^2} B(|v-v_{*}|, \cos \theta) f(t,v_*)g(t,v)(\phi(v')-\phi(v))\: dn  dv_*dv,\label{weak_colision_operator}
\end{align} where
\begin{align} \label{def_theta}
     \cos \theta = \left( \frac{v_*-v}{|v_*-v|}, n \right), 
     \quad \theta \in \left[0,\frac{\pi}{2} \right].
\end{align} 
In the case $f=g$ in \eqref{weak_colision_operator}, we have
\begin{align} \label{weak_4}
\begin{split}
     \int_{\mathbb{R}^3} Q(f,f)(t,v) \phi(v) \; dv 
     &= \frac{1}{2}\int_{\mathbb{R}^3}\int_{\mathbb{R}^3}\int_{\mathbb{S}_{+}^2}  B(|v-v_{*}|, \cos \theta) f(t,v_*)f(t,v)\\
     &\quad\quad\quad \times(\phi(v')+\phi(v_*')-\phi(v)-\phi(v_*))\: dn dv_*dv.
\end{split}
\end{align}
Taking $\phi(v)=1,|v|^2,$ and $ v_i$ into above equation, for any $t>0$, we have
\begin{align} \label{conser f0}
    \int_{\mathbb{R}^3} f(t,v) \;dv = \int_{\mathbb{R}^3} f_0(v) \;dv, 
    \text{\quad and  \quad}  \int_{\mathbb{R}^3} f(t,v)|v|^2 \;dv \leq \int_{\mathbb{R}^3} f_0(v)|v|^2 \;dv 
\end{align} by \eqref{loss_energy}, and
\begin{align} \label{conser f0 momentum}
    \int_{\mathbb{R}^3} f(t,v)(v \cdot e_i) \; dv = \int_{\mathbb{R}^3} f_0(v)(v \cdot e_i) \;dv 
\end{align} for $i=1,2,3$. Here $e_1=(1,0,0),\;e_2=(0,1,0)$ and $e_3=(0,0,1)$. In this paper, we suppose
\begin{align} \label{initial contidion}
    \int_{\mathbb{R}^3} f_0(v) \; dv = 1,
    \quad \int_{\mathbb{R}^3} f_0(v) (v \cdot e_i) \; dv = 0, 
    \quad \int_{\mathbb{R}^3} f_0(v)|v|^2 \; dv = 1
\end{align} for $i=1,2,3$, without loss of generality. See also \cite{GPV2004, SC2006, SC2009, SCM2006} in which the authors consider the same initial conditions \eqref{initial contidion}. By \eqref{conser f0}-\eqref{initial contidion}, it holds that 
\begin{align} \label{initial_bound}
\begin{split}
    &\|f(t,v)\|_{1,0}=\|f_0(v)\|_{1,0}=1, \quad\|f(t,v)\|_{1,2} \leq\|f_0(v)\|_{1,2} < +\infty\\
    &\quad\quad\quad \int_{\mathbb{R}^3} f_0(v) (v \cdot e_i) \; dv=\int_{\mathbb{R}^3} f(t,v) (v \cdot e_i) \; dv = 0 
\end{split}
\end{align} for $t>0$ and $f_0 \in L^1_2(\mathbb{R}^3)$. Moreover, 
the initial entropy is denoted as
\begin{align} \label{def_H0}
    \int_{\mathbb{R}^3} f_0(v) \log f_0(v) \; dv = H_0.
\end{align}

Throughout the paper, we will consider a hard sphere model 
\begin{align*}
  B(|v-v_{*}|, \cos \theta) = |(v-v_*) \cdot n |=|v-v_*| |\cos \theta|
\end{align*} for $\theta \in [0,\frac{\pi}{2}]$. Since $\delta(v)$ is approximated by $C^\infty_c$ test function, we can take $\phi(v)=\delta(v-v_1)$ in \eqref{weak_colision_operator}. By integrating about $v$, we have
\begin{align*}
       Q(f,g)(t,v_1) &= \int_{\mathbb{R}^3}\int_{\mathbb{R}^3}\int_{\mathbb{S}_{+}^2} |(v-v_*)\cdot n| f(t,v_*)g(t,v)(\delta(v'-v_1)-\delta(v-v_1)) \: dn  dv_*dv \\
       &:= Q^{+}(f,g)(t,v_1)-g(t,v_1) L f(t,v_1).
\end{align*} By changing variables $(v,v_*,n) \rightarrow (v',v_*',-n)$ with Jacobian determinant $\alpha$, and \eqref{alpha}, 
 \begin{align} \label{def_Q+}
 \begin{split}
     Q^{+}(f,g)(t,v_1) &= \int_{\mathbb{R}^3}\int_{\mathbb{R}^3}\int_{ \mathbb{S}_{+}^2}  |(v-v_*)\cdot n| f(t,v_*)g(t,v)\delta(v'-v_1)\: dn  dv_*dv  \\
     &= \frac{1}{\alpha^2}  \int_{\mathbb{R}^3}\int_{\mathbb{R}^3}\int_{\mathbb{S}_{+}^2} |(v-v_*)\cdot n| f(t,{'v}_*)g(t,{'v})\delta(v-v_1)\: dn  dv_*dv \\
     &= \frac{1}{\alpha^2} \int_{\mathbb{R}^3}\int_{\mathbb{S}_{+}^2}  |(v_1-v_*)\cdot n| f(t,{'v}_*)g(t,{'v}_1)\: dn  dv_*.
 \end{split}
\end{align} And we have
\begin{align} \label{def_Lf}
\begin{split}
      Lf(t,v_1) &= \int_{\mathbb{R}^3}\int_{\mathbb{S}_{+}^2} |(v_1-v_*)\cdot n| f(t,v_*)\: dn 
    dv_* \\
    &= \pi \int_{\mathbb{R}^3} |v_1-v_*| f(t,v_*)\; dv_*.
\end{split}
\end{align} Before, we parameterize $v'$ and $v'_*$ in two different ways by using impact direction $n \in \mathbb{S}_{+}^2$ in \eqref{v'}, and $\sigma\in \mathbb{S}^2$ in \eqref{v'_sigma}. In Figure \ref{fig_v'}, there are $n$ and $\sigma$, and we get
\begin{align} \label{theta-2theta}
    \frac{v_*-v}{|v_*-v|} \cdot \sigma = \cos 2\theta
    \text{\quad where \quad} \frac{v_*-v}{|v_*-v|} \cdot n = \cos \theta.
\end{align} If we choose and fix the direction of $v_*-v$ as the $z$-axis, we can consider spherical coordinate of
\begin{align}\label{coordinate_n,sigma}
    n = (\cos \phi \sin \theta, \sin \phi \sin \theta, \cos \theta), 
    \quad \sigma = (\cos \phi \sin 2\theta, \sin \phi \sin 2\theta, \cos 2\theta)
\end{align} for $\phi \in [0,2\pi], \;\theta \in \left[0,\frac{\pi}{2} \right]$. By using $\sigma$ of instead $n$, we can focus the amplitude, $|v-v_*|$ in $Q^{+}(f,f)(t,v_1)$ without $\cos \theta$ term in \eqref{def_Q+}. Moreover, we define the rescaled $\sigma'$ as
\begin{align} \label{def_simga'}
    \sigma'= \frac{\beta}{2}|v-v_*| \sigma,
\end{align} which is used in the proof of Lemma \ref{Q+_E_estimate}. Using
\begin{align} \label{jacobian_n,simga}
    d \sigma = 4 \cos \theta \; d n \text{\quad and \quad } d\sigma' = \frac{\beta^2}{4}|v-v_*|^2 d\sigma, 
\end{align} the weak form of the $Q^{+}(f,g)(v)$ can be expressed in the variables $\sigma$ and $\sigma'$ as follows:
\begin{align}
    &\int_{\mathbb{R}^3} Q^{+}(f,g)(t,v) \phi(v) \; dv  \notag \\
    &= \int_{\mathbb{R}^3}\int_{\mathbb{R}^3}\int_{\mathbb{S}_{+}^2}  |(v-v_*)\cdot n| f(t,v_*)g(t,v)\phi(v')\: dn  dv_*dv \notag \\
    &=\frac{1}{4} \int_{\mathbb{R}^3}\int_{\mathbb{R}^3}\int_{\mathbb{S}^2}  |v-v_*| f(t,v_*)g(t,v)\phi(v') \: d\sigma dv_*dv \label{Q+_weak_sigma} \\
    &= \frac{1}{\beta^2} \int_{\mathbb{R}^3}\int_{\mathbb{R}^3}\int_{ B\left(\frac{\beta}{2}|v-v_*| \right)}  |v-v_*|^{-1} f(t,v_*)g(t,v)\phi(v') \: d\sigma' dv_*dv, \label{Q+_weak}
\end{align} where
\begin{align*}
    B\left(\frac{\beta}{2}|v-v_*| \right)=\left\{ x \in \mathbb{R}^3 : \;|x|= \frac{\beta}{2}|v-v_*| \text{\quad for \quad} v, v_* \in \mathbb{R}^3 \right\}.
\end{align*} 

\subsubsection{Carleman type representation for
granular gases} In 1957, Carleman introduced a new representation of the gain term collision operator for the Boltzmann equation with a cutoff kernel. When \(v'\) and \(v\) are fixed, the set of possible vectors for \(v_*'\) forms a hyperplane orthogonal to \(v - v'\) and passing through \(v\) (see Figure \ref{fig_v'_elastic}). Using this fact, he obtained the gain term integrated over the variables \(v'\) and \(v_*'\), where the range is in \(\mathbb{R}^3\) and a hyperplane for each, as described in \cite{T1957}. Moreover, there is a Carleman-type representation for granular gas presented in \cite[Proposition 1.5]{SC2006} by Mischler and Mouhot. To obtain a Carleman-type representation, we define a vector \(P\) as:
\begin{align*} 
    P &= \frac{1}{\beta}v'-\left(\frac{1}{\beta}-1\right)v,
\end{align*} which is located on the extension line of $v$ and $v'$ in ratio 
\begin{align*}
    |P-v|:|P-v'|=1:1-\beta
\end{align*} for $\beta \neq 1$ as in Figure \ref{fig_v'}. For $\beta=1$ (elastic collisions), the vector $P$ equals $v'$. Then $P$ also satisfies $ P-v\;\bot\;P-v_*$.
Next, we introduce a 2-dimensional hyperplane $E^{v'-v}_P$ which is normal to the vector $v'-v$ around the vector $P$; it is defined as 
\begin{align} \label{E_pvv}
\begin{split}
        E^{v'-v}_P =\bigg \{x \in \mathbb{R}^3 : \; (x-P)\;\bot \;(v-v') \text{\; where \;} P=\frac{1}{\beta}v'-\left(\frac{1}{\beta}-1\right)v\;\bigg \} \subseteq \mathbb{R}^2.
\end{split}
\end{align} We notice that if $v$ and $v'$ are fixed, $v_*'$ is in the hyperplane $E^{v'-v}_P$.  Then $Q^{+}(f,g)(t,v_1)$ is expressed by
\begin{align} \label{Q+_Cal}
\begin{split}
        &Q^{+}(f,g)(t,v_1) \\
    &= \int_{\mathbb{R}^3} Q^{+}(f,g)(t,v)\delta(v-v_1) \; dv\\
    &= \frac{1}{\beta^2} \int_{\mathbb{R}^3}\int_{\mathbb{R}^3} \delta(v-v_1) g(t,{'v}) |v-{'v}|^{-1}\int_{{'v}_* \in E^{v-{'v}}_P} f(t,{'v}_*) \;dE_{{'v}_*} d{'v} dv \\
    &= \frac{1}{\beta^2} \int_{\mathbb{R}^3}\int_{\mathbb{R}^3} \delta(v'-v_1) g(t,v) |v'-v|^{-1}\int_{v_* \in E^{v'-v}_P} f(t,v_*) \;dE_{v_*} dv dv',
\end{split}
\end{align} 
where $dE_{{'v}_*}$ and $dE_{v_*}$ are the Lebesgue measures on the plane $E^{v-{'v}}_P$ and $E^{v'-v}_P$, respectively. This Carleman-type representation is used crucially for the estimates of the gain term, $Q^{+}(f,f)$, through the paper. 

\subsection{Main results}
We are now ready to state our main theorems. Given $t>0$ and $v \in \mathbb{R}^3$, we prove that the solution \(f(t,v)\) is bounded pointwise from above by \(C_{1,f_0}(1+t)^3\) in the region including zero velocity, and by \( C_{2,f_0}(1+t)\), away from the zero velocity ($v \neq 0$).
We have also derived an upper bound that depends on the coefficient of normal restitution constant $\alpha \in (0,1]$ and the time $t>0$. We notice this upper bound becomes smaller as $\alpha \rightarrow 1$, and it becomes a constant when $\alpha =1$.

Moreover, the first part of the minimum in \eqref{thm_v except 0} indicates that the density distribution function during the cooling process is very sensitive to the temporal variable when the velocity is near zero. The growth of this bound can be considered to be natural, as shown in Figure \ref{graph}, during the cooling process. In addition, we observe that there is a loss of weight in the \(L^{\infty}\) propagation in the first part of the minimum in \eqref{thm_v include 0}.
Here, we define the cooling time $T_c$ in the cooling process as 
    \begin{align} \label{Tc}
    \begin{split}
      T_c :&= \inf {\{T \geq 0, \;\mathcal{E}(t)=0 \;\; \forall t>T\}}\\
        &=\sup {\{ S \geq 0, \;  \mathcal{E}(t)>0 \;\;  \forall t<S \}},
    \end{split}
    \end{align} where $\mathcal{E}$ is the total kinetic energy, which is defined as $\mathcal{E}(t)= \int_{\mathbb{R}^3} f(t,v)|v|^2 dv$. 

\begin{theorem} \label{L_infty_thm} 
Let $0 <\alpha \leq 1$ be the coefficient of normal restitution constant, which satisfies \eqref{alpha}. Assume that \eqref{initial contidion} holds and $f_0 \in L_{s}^{\infty}$ for $s>2$. 

\noindent(i) Then there exist positive constants $C_1(\|f_0\|_{\infty,s})$ and $C_2$ such that 
    \begin{align} \label{thm_v include 0}
        f(t,v) \leq C_1(\|f_0\|_{\infty,s})\min{\left\{ \frac{(1+t)^3}{1+|v|},  \frac{e^{C_2\left(\frac{1}{\alpha^2}-1\right)t}}{(1+|v|)^2}\right\}}
    \end{align} for almost every $v \in \mathbb{R}^3$ and for every $t>0$. 
    
\noindent (ii)   Moreover, there exists a positive constant $C_3(\|f_0\|_{\infty,s})$ such that
      \begin{align} \label{thm_v except 0}
        f(t,v) \leq C_3(\|f_0\|_{\infty,s})\min{\left\{ \frac{1+t}{|v|^2},  \frac{e^{C_2\left(\frac{1}{\alpha^2}-1\right)t}}{(1+|v|)^2}\right\}}
    \end{align} for almost every $v \neq 0 \in \mathbb{R}^3$ and for every $t>0$. 
\end{theorem}

\begin{corollary} \label{Tc_Cor}
    Assume that \eqref{initial contidion} holds and $f_0 \in L_{s}^{\infty}$ for $s>2$. Then the cooling time $T_c$ of a solution, defined in \eqref{Tc}, is infinite.
\end{corollary}

\begin{remark}
Theorem \ref{L_infty_thm} holds even for the infinite energy case as long as we have $f(t,\cdot)\in L^1_1$ for any $t\ge 0$; i.e., the finiteness condition on the initial energy in \eqref{initial contidion} can be replaced by $f(t,\cdot)\in L^1_1$ for $t\ge 0.$ Either one of these can guarantee that $Lf$ be well-defined in Lemma \ref{Lf>|v|}.
\end{remark}

\begin{remark} \label{remark_Tc}
    In \cite{SCM2006}, Mischler, Mouhot and Ricard studied the spatially homogeneous inelastic Boltzmann equation for hard spheres, considering a general form of collision rate that includes variable restitution coefficients depending on the kinetic energy and the relative velocity. In Theorem 1.2, they consider the case where the collision rate B is independent of the kinetic energy. In particular, according to Theorem 1.2-(i) and Remark 1.3-(1), for inelastic hard spheres with a constant restitution coefficient, which we studied in this paper, there exists a unique solution under the assumption of \eqref{initial contidion}, and the cooling time of the solution is infinite. In Corollary \ref{Tc_Cor}, we obtain cooling time $T_c$ is infinite, but we also assume $f_0 \in L^{\infty}_s$.
 \end{remark}

Many researchers are interested in the upper Maxwellian bound, lower Maxwellian bound, and the convergence of the system to the Maxwellian state in the elastic Boltzmann equation, where $\alpha$ is 1. Due to the loss of energy in inelastic collisions, we also expect the solution to exhibit thinner velocity tails in the inelastic Boltzmann equation and to approach the Maxwellian upper bound for sufficiently large velocities for all \(0 < \alpha \leq 1\). One of the most significant aspects of Theorem \ref{L_infty_thm} is that it provides pointwise upper bounds, allowing us to predict the behavior of the solution. Using these upper bound estimates, we will establish the Maxwellian upper bound.

\begin{theorem} \label{thm_upper}
Let $0 <\alpha \leq 1$ be the coefficient of normal restitution constant, which satisfies \eqref{def_alpha_2}. We assume that \eqref{initial contidion} holds, and that $f_0(v) \leq M_0(v)$ for almost every $v \in \mathbb{R}^3$, where $M_0(v) = e^{-a_0|v|^2+c_0}$. Then there exist $a \in (0,a_0), \;b \in \mathbb{R}$ and $C>0$, depending on $a_0$ and $c_0$, such that
    \begin{align} \label{upper_maxwell}
        f(t,v) \leq \min{\left\{(1+t)^3, e^{C\left(\frac{1}{\alpha^2}-1\right)t} \right\}} e^{-a|v|^2+b}
    \end{align} for almost every $v \in \mathbb{R}^3$ and every $t > 0$.
\end{theorem}

\begin{remark}
Fix the coefficient of normal restitution constant $0<\alpha<1$, and denote $T_\alpha = \alpha^2/(1-\alpha^2)$. Then, by Theorem \ref{L_infty_thm}-(i), there is a positive constant $\overline{C}(\|f_0\|_{\infty,s})$ such that
  \begin{align*} 
          \esssup_{v \in \mathbb{R}^3 } f(t,v)  (1+|v|)^2\leq \overline{C}(\|f_0\|_{\infty,s})
    \end{align*} for $0<t<T_\alpha$, under \eqref{initial contidion} and that $f_0 \in L_{s}^{\infty}$ for $s>2$. By Theorem \ref{thm_upper}, there exist $\overline{a} \in (0,a_0), \;\overline{b} \in \mathbb{R}$, depending on $a_0$ and $c_0$, such that
     \begin{align*}
        f(t,v) \leq  e^{-\overline{a}|v|^2+\overline{b}}
    \end{align*} for almost every $v \in \mathbb{R}^3$ and $0<t<T_\alpha$ under $f_0(v) \leq M_0(v)$.
\end{remark}
\begin{remark} \label{thm_remark}
    In \cite{SCM2006}, in Proposition 3.2-(ii), Mischler, Mouhot, and Ricard proved that there exist $C, r'>0$ such that
    \begin{align*}
        \sup_{t\in[0,T_c)} \int_{\mathbb{R}^N} f(t,v) e^{r'|v|^{\eta}} dv \leq C
    \end{align*} if $f_0(v)e^{r|v|^{\eta}} \in L^{1}_0(\mathbb{R}^N)$ for $r>0$ and $\eta \in (0,2]$, and $T_c$ is defined as \eqref{Tc}. 
\end{remark} 

Under elastic collisions, the total entropy remains finite for all \(t > 0\) by the \(H\)-theorem if it is initially finite. As a consequence, it establishes a uniform lower bound for \(Lf(t,v)\) as stated in Lemma 4 of \cite{L1983}, which has been used to provide a pointwise upper-bound of solution via the Gr\"onwall inequality as in \cite{L1983}. On the other hand, we observe that the inelastic collisions lack an \(H\)-theorem and a uniform lower bound for \(Lf(t,v)\). Thus, we derive two weaker versions for lower bounds of $Lf(t,v)$ in Lemma \ref{Lf>|v|}. Since these lower bounds are not uniform with respect to velocity and time, $L^{\infty}$ norms in Theorem \ref{L_infty_thm} depend also on time.

\subsection{Outline of the rest of the article}
Here we briefly outline the structure of the remainder of the paper. 
In Section \ref{sec.Linfty}, we prove \(L^{\infty}\) estimates. First, we obtain the upper bound of the gain term of the collision operator, \(Q^{+}(f,f)\), in Lemma \ref{Q+_estimate} through Lemma \ref{f_esti_lemma}. Next, we apply Grönwall's lemma in Lemma \ref{sup_f} to the Boltzmann equation to prove Theorem \ref{L_infty_thm}. In Section \ref{sec.Maxupper}, based on a comparison principle in Lemma \ref{comparison}, we extend \(L^{\infty}\) estimates to establish the Maxwellian upper bound.

\section{\texorpdfstring{$L^{\infty}$}{} estimates for the inelastic Boltzmann equation }
\label{sec.Linfty}

In this Section, we estimate the \(L^{\infty}\) norm of the solution to the inelastic Boltzmann equation. Let $0 <\alpha \leq 1$ be the coefficient of normal restitution constant, which satisfies \eqref{def_alpha_2}. The following Lemma \ref{sup_f}, which is derived from Grönwall's inequality, will be repeatedly used in the paper. \\

\begin{lemma}\label{sup_f}
Let $h_1(t)$ and $h_2(t)$ be continuous real functions on $\mathbb{R}_{+}$ and $h_1(t)>0$. If
    \begin{align*}
        \frac{d}{dt}f + h_1 f \leq h_2
    \end{align*} for $t>0$, then
    \begin{align*}
         f(t) \leq  \sup_{0 \leq s \leq t} \frac{h_2(s)}{h_1(s)} + f(0)
    \end{align*} for $t>0$.
\end{lemma}
\begin{proof}
From
\begin{align*}
    \frac{d}{dt}\left( e^{\int^t_0 h_1(s) ds} f\right) \leq 
     e^{\int^t_0 h_1(s) ds} h_2(t),
\end{align*} we obtain
\begin{align*}
    f(t) &\leq \int^t_0 e^{-\int^t_\tau h_1(s) ds} h_2(\tau) \; d\tau + f(0) \\
    &\leq \left(1-e^{-\int^t_0 h_1(s) ds} \right)\sup_{0 \leq s \leq t} \frac{h_2(s)}{h_1(s)}  + f(0) \leq  \sup_{0 \leq s \leq t} \frac{h_2(s)}{h_1(s)} + f(0)
\end{align*} for $t> 0$.
\end{proof}

\begin{lemma} \label{entropy_lemma}
Assume that \eqref{initial contidion} holds, and that the initial entropy $H_0$ in \eqref{def_H0} is finite. Then
there exists a positive constant $C$ that depends only on $\|f_0\|_{1,2}$, such that
    \begin{align} \label{entropy_under_t}
       \int_{\mathbb{R}^3} f(t,v) \log f(t,v) dv \leq C t\left(  \frac{1}{\alpha^2} -1  \right) + H_0
    \end{align} for $t>0$.
\end{lemma}
\begin{proof}
  We denote
$f:=f(t,v),\;f_{*}:=f(t,v_*),\;f':=f(t,v'),\;f_{*}':=f(t,v_*')$. From the weak form of the collision operator \eqref{weak_4}, we have
    \begin{align}
           &\frac{d}{dt}\int_{\mathbb{R}^3} f \log f \;dv= \int_{\mathbb{R}^3}  Q(f,f)(t,v) \log f \;dv \notag \\
        &= \frac{1}{2} \int_{\mathbb{R}^3} \int_{\mathbb{R}^3} \int_{\mathbb{S}_{+}^2}
        f f_{*} \log \frac{f'f_*'}{f f_{*}} |(v-v_*) \cdot n | \;dn dv dv_* \notag \\
        &= \frac{1}{2}\int_{\mathbb{R}^3} \int_{\mathbb{R}^3} \int_{\mathbb{S}_{+}^2} ff_{*}
        \left(  \log \frac{f'f_*'}{f f_{*}} - \frac{f'f_*'}{f f_{*}} +1   \right)|(v-v_*) \cdot n | \;dn dv dv_* \notag \\
        &\quad + \frac{1}{2}\int_{\mathbb{R}^3} \int_{\mathbb{R}^3} \int_{\mathbb{S}_{+}^2}
        \left(  f'f_*'-f f_*  \right)|(v-v_*) \cdot n |\; dn dv dv_*. \label{alpha_term} 
    \end{align} 
    Taking a change of variables $(v,v_*,n) \rightarrow (v',v_*',-n)$ with the Jacobian determinant $\alpha$ and \eqref{alpha}, we have
    \begin{align} \label{entropy_change of variable}
    \begin{split}
        &\int_{\mathbb{R}^3} \int_{\mathbb{R}^3} \int_{\mathbb{S}_{+}^2}
        f'f_*'|(v-v_*) \cdot n |\; dn dv dv_*\\
        &=  \frac{1}{\alpha^2}\int_{\mathbb{R}^3} \int_{\mathbb{R}^3} \int_{\mathbb{S}_{+}^2}
        ff_*|(v-v_*) \cdot n | \;dn dv dv_*.
    \end{split}
    \end{align} By using $\log x -x + 1 \leq 0$ for $x>0$, and applying \eqref{entropy_change of variable} to \eqref{alpha_term}, we have
    \begin{align} \label{int_entropy}
         &\frac{d}{dt}\int_{\mathbb{R}^3} f \log f \;dv \leq 
         \frac{1}{2} \left(  \frac{1}{\alpha^2} -1  \right)\int_{\mathbb{R}^3} \int_{\mathbb{R}^3} \int_{\mathbb{S}_{+}^2}
        ff_*|(v-v_*) \cdot n | \;dn dv dv_*.
    \end{align} By the assumption \eqref{initial contidion} on the profile $f_0$, we obtain
    \begin{align*}
        \int_{\mathbb{R}^3} \int_{\mathbb{R}^3} \int_{\mathbb{S}_{+}^2} 
        ff_*|(v-v_*) \cdot n | \;dn dv dv_* &\lesssim 1.
    \end{align*} Lastly, by integrating \eqref{int_entropy} with respect to $t>0$, we obtain
    \begin{align*}
        \int_{\mathbb{R}^3} f(t,v) \log f(t,v)\; dv \leq C t\left(  \frac{1}{\alpha^2} -1  \right) + H_0
    \end{align*} for some positive constant $C$ that depends only on $\|f_0\|_{1,2}$.
\end{proof}

\begin{remark}
The quantitative variants of Boltzmann's H-theorem were proven in \cite{MR1700142, MR1964379}. In \cite{DV2005}, by using H-theorem, convergence to equilibrium for solutions of the inhomogeneous Boltzmann equation was studied. The H-theorem was also proven for the spatially homogeneous relativistic Boltzmann equation in \cite{MR4271957, MR3166961}.
\end{remark}

We first estimate $Lf(t,v) $ in \eqref{def_Lf}, and later, the lower bound of $Lf(t,v)$ will serve the role of $h_1(t)$ in Lemma \ref{sup_f} for the Boltzmann equation. 

\begin{lemma}\label{Lf>|v|}
Assume that \eqref{initial contidion} holds, and that the initial entropy $H_0$ in \eqref{def_H0} is finite. Then there exist positive constants $C_1$ and $C_2$ such that
    \begin{align} \label{Lf_3}
        Lf(t,v_1) \geq \max{\left\{\pi|v_1|, \;C_1 e^{-C_2 \left(\frac{1}{\alpha^2}-1 \right)t}(1+|v_1|)\right\}}
    \end{align} for any $v_1 \in \mathbb{R}^3$ and $t>0$.
\end{lemma}

\begin{proof}
Since $\phi(s)=|s|$ is a convex function, by Jensen's inequality, we have
    \begin{align} \label{Lf_1}
    \begin{split}
        Lf(t,v_1)&= \int_{\mathbb{R}^3} \int_{\mathbb{S}_{+}^2} |(v_1-v_*)\cdot n|f(t,v_*) \;dndv_* = \pi \int_{\mathbb{R}^3}  |v_1-v_*|\;f(t,v_*) \; dv_* \\
        &\geq \pi \left|\sum_{i=1,2,3}\left\{(v_1 \cdot e_i) \int_{\mathbb{R}^3} f(t,v_*)\; dv_* -\int_{\mathbb{R}^3} (v_* \cdot e_i) f(t,v_*) \; dv_* \right\} e_i\right|\\
        &=\pi \left|\sum_{i=1,2,3}\left\{(v_1 \cdot e_i) \times 1 - 0\right\} e_i\right|
        = \pi |v_1|
      \end{split}
    \end{align} for $v_1 \in \mathbb{R}^3$. 
    
On the other hand, by Lemma \ref{entropy_lemma}, we have
\begin{align}
    &\int_{\{|v_1-v_*|<r\}} f(t,v_*) \; dv_* \notag \\
    &\leq \int_{\{|v_1-v_*|<r,\;f(t,v_*)\leq j\}} f(t,v_*) \; dv_* + \int_{\{|v_1-v_*|< r,\;f(t,v_*)>j\}} f(t,v_*) \; dv_* \notag \\
    &\leq \frac{4\pi}{3} j r^3  +  \frac{1}{\log j } \left( C t\left(  \frac{1}{\alpha^2} -1  \right) + H_0\right)\label{r,j}
\end{align} for $r, j >0$. Taking
\begin{align*}
    \log j = \frac{4}{\|f_0\|_{1,0}}\left( C t\left(  \frac{1}{\alpha^2} -1  \right) + H_0\right),  \quad  r = \left(\frac{3}{8\pi} \|f_0\|_{1,0}  \frac{1}{j}  \right)^{\frac{1}{3}}
\end{align*} in \eqref{r,j}, we find that
\begin{align*}
      &\int_{\{|v_1-v_*|<r\}} f(t,v_*) \; dv_* \leq \frac{1}{2}\|f_0\|_{1,0} .
\end{align*} Then we have
\begin{align} \label{Lf>e}
\begin{split}
    L(f)(t,v_1) &= \pi \int_{\mathbb{R}^3} f(t,v_*)|v_1-v_*| \; dv_1 \geq \pi r\int_{\{|v_1-v_*|\geq r\}} f(t,v_*) \; dv_*         \\
    &=\pi r\left( \int_{\mathbb{R}^3} f(t,v_*) \; dv_* - \int_{\{|v_1-v_*|<r\}} f(t,v_*) \; dv_*    \right) \\
    &\geq \frac{1}{2}\pi r \|f_0\|_{1,0} = \frac{\pi}{2} \left(\frac{3}{4\pi} \right)^{\frac{1}{3}} \|f_0\|_{1,0}^{\frac{4}{3}}e^{-\frac{4}{3 \|f_0\|_{1,0}}\left( C t\left(  \frac{1}{\alpha^2} -1  \right) + H_0\right)}\\
    &\geq c_1 e^{-c_2\left(\frac{1}{\alpha^2}-1 \right)t}
\end{split}
\end{align} 
for some positive constants $c_1$ and $c_2$. From \eqref{Lf_1}, we obtain
\begin{align*}
    L(f)(t,v_1) \geq |v_1| \geq |v_1| e^{-c_2\left(\frac{1}{\alpha^2}-1 \right)t}
\end{align*}
since $e^{-x} \leq 1$ for $x \geq 0$. Then, by \eqref{Lf>e}, we obtain
\begin{align*}
     L(f)(t,v_1) \geq \frac{1}{2}\left( c_1+|v_1| \right)e^{-c_2\left(\frac{1}{\alpha^2}-1 \right)t} \geq C_1(1+|v_1|)e^{-c_2\left(\frac{1}{\alpha^2}-1 \right)t}
\end{align*}
for some positive constants $C_1$ and $c_2$.
\end{proof}

\begin{remark}
    The finiteness condition on the initial energy in \eqref{initial contidion} can be replaced by $f(t,\cdot)\in L^1_1.$ Either of these can guarantee that $Lf$ be well-defined.
\end{remark}

\begin{remark}
    In \cite{L1983}, in Lemma 4, Arkeryd proved that
    \begin{align*}
        Lf(t,v) \geq C(C_0, \|f\|_{1,2}, \|f\|_{1,0})(1+|v|)^{\gamma}, \quad \gamma \in (0,1]
    \end{align*} under the conditions $f \in L^1_2$ and $\int_{\mathbb{R}^3}f \log f \; dv < C_0$ for the homogeneous Boltzmann equation for hard potentials with angular cutoff. For the hard spheres case, $\gamma$ is 1. In elastic collisions, it holds that $\int_{\mathbb{R}^3}f \log f \; dv  \leq \int_{\mathbb{R}^3}f_0 \log f_0 \; dv$ by the H-theorem. However, we cannot control $\int_{\mathbb{R}^3}f \log f \; dv $ in inelastic collisions, so $Lf(t,v) \gtrsim 1$ does not hold. 
\end{remark}

\begin{lemma}\label{Q+_estimate} Let $f(t,\cdot) \in L_{0}^{1}$. For any given fixed $v\in\mathbb{R}^3$ and $t>0$, we have
\begin{align*}
    \left\| \int_{\mathbb{R}^3} Q^{+}(f,f)(t,v_1)|v_1-v|^{-1}\;dv_1 \right\|_{\infty,0} \; \leq \;  4\pi\|f\|_{1,0}^2.
\end{align*} 
\end{lemma}
\begin{proof}
    In \eqref{Q+_weak_sigma}, we replace $v, v'$ with $v_1, v_1'$ respectively. Next, we choose \textcolor{blue}{$\phi(v_1)=|v-v_1|^{-1}$} for a given fixed $v\in\mathbb{R}^3$. Using \eqref{jacobian_n,simga}, we have
    \begin{align*}
        &\int_{\mathbb{R}^3} Q^{+}(f,f)(t,v_1)|v_1-v|^{-1} \; dv_1 \\
        &= \frac{1}{4} \int_{\mathbb{R}^3}\int_{\mathbb{R}^3}  f(t,v_1)f(t,v_*) |v_1-v_*| \int_{\mathbb{S}^2} |v_1'-v|^{-1} \; d\sigma dv_* dv_1  \\
        &= \frac{1}{2\beta}\int_{\mathbb{R}^3}\int_{\mathbb{R}^3} f(t,v_1)f(t,v_*) \int_{\mathbb{S}^2} |\sigma + A|^{-1} \; d\sigma dv_* dv_1  \;\; \leq \;\; 4\pi\|f\|_{1,0}^2,
    \end{align*} where $A$ is $\big(\frac{\beta}{2}|v_1-v_*|\big)^{-1}\big(\frac{v_1+v_*}{2}+\frac{1-\beta}{2}(v_1-v_*)-v\big)$ from \eqref{v'_sigma}$_1$.
\end{proof}

\begin{lemma} \label{lim_int_Q+}
     For any $v, \bar{v}\in\mathbb{R}^3$, we recall  $E_{P}^{v-\bar{v}}$ from \eqref{E_pvv}. We let $d_{v_1}\ge 0$ be the distance from $v_1\in\mathbb{R}^3$ to the plane $E_{P}^{v-\bar{v}}$, and
      \begin{align} \label{def_D_j}
        D_{j}(v_1) = \sqrt{\frac{j}{\pi}} e^{-(\sqrt{j}d_{v_1})^2}.
    \end{align} Then we have
    \begin{align*}
         \lim_{j \rightarrow \infty} \int_{\mathbb{R}^3} D_j (v_1) Q^{+}(f,f)(t,v_1) \; dv_1 =\int_{v_1 \in E_{P}^{v-\bar{v}}} Q^{+}(f,f)(t,v_1) \; dE_{v_1}.
    \end{align*}
\end{lemma}
\begin{proof} 
Note that $D_j(v_1)$ converges to $\delta(v_1)$ in distribution as $j \rightarrow \infty$. Let $a$ be the vector in $E_{P}^{v-\bar{v}}$ satisfying $dv_1 = da\;d\left(d_{v_1} \right)$. We obtain
     \begin{align*}
         &\lim_{j \rightarrow \infty} \int_{\mathbb{R}^3} D_j (v_1) Q^{+}(f,f)(t,v_1) \; dv_1 \\
         &=\lim_{j \rightarrow \infty} \int_{-\infty}^{+\infty} D_j (v_1) \int_{a \in E_{P}^{v-\bar{v}} }Q^{+}(f,f)(t,v_1) \; da\;d \left(d_{v_1} \right) \\
         &=\int_{v_1 \in E_{P}^{v-\bar{v}}} Q^{+}(f,f)(t,v_1) \; dE_{v_1}.
    \end{align*}
\end{proof}
To estimate the gain term collision operator $Q^{+}(f,f)(t,v)$, we will establish the following Lemma \ref{Q+_E_estimate} and Lemma \ref{f_esti_lemma}. Write a Carleman-type representation of $Q^{+}(f,f)$ from in \ref{Q+_Cal}, 
\begin{align*} 
        Q^{+}(f,f)(t,v_1)
    = \frac{1}{\beta^2} \int_{\mathbb{R}^3}\int_{\mathbb{R}^3} \delta(v'-v_1) f(t,v) |v'-v|^{-1}\int_{v_* \in E^{v'-v}_P} f(t,v_*) \;dE_{v_*} dv dv',
\end{align*} where $ E^{v'-v}_P$ is in \eqref{E_pvv}. We will estimate 
\begin{align*}
    \int_{v_* \in E^{v'-v}_P} f(t,v_*) \;dE_{v_*} 
    \text{\quad and \quad }   
      \int_{\mathbb{R}^3}  f(t,v_1)  |v-v_1|^{-1} \; dv_1
\end{align*} in Lemma \ref{Q+_E_estimate} and Lemma \ref{f_esti_lemma}, respectively.

\begin{lemma} \label{Q+_E_estimate}
Let $f \in L_{0}^{1}$. We have
\begin{align*}
    \int_{v_1 \in E_{P}^{v-\bar{v}}} Q^{+}(f,f)(t,v_1)\; dE_{v_1} \leq 2\pi\|f\|_{1,0}^2
\end{align*}  for any $v, \bar{v} \in \mathbb{R}^3$ and $t>0$. 
\end{lemma}
\begin{proof}
     In Figure \ref{fig_v'}, for given $v_1,v_*$ (with $v_1$ replacing $v$ in Figure \ref{fig_v'} here), the two spheres are given where the sphere on the left represents the possible orbit for $v_1'$ and the other on the right represents the one for $v_*'$. Let us denote the sphere on the left as $S(v_1')$. Then we denote $\chi=1$ if the plane $E_{P}^{v-\bar{v}}$ and $S(v_1')$ intersect and $\chi=0$ if the plane $E_{P}^{v-\bar{v}}$ and $S(v_1')$ do not intersect. For $D_j(v'_1)$ in \eqref{def_D_j}, if $\chi=1$, note that we have
      \begin{align} \label{d_j_lessim}
     \begin{split}
         &\int_{S(v_1')} D_j (v'_1) \; d\sigma' = \int_{-\infty}^{+\infty} D_j(v'_1) \int_{a \in S(v_1')\cap E_{P}^{v-\bar{v}}} \; da\left(d_{v_1} \right)\\
         &\leq \int_{a \in B\left(\frac{\beta|v_1-v_*|}{2}\right)}1 \; da \leq \pi \beta |v_1-v_*|,
     \end{split}
     \end{align} where
     \begin{align*}
         B\left(\frac{\beta}{2}|v_1-v_*| \right)=\left\{ x \in \mathbb{R}^2 : \;|x|= \frac{\beta}{2}|v_1-v_*| \text{\quad for \quad} v_1, v_* \in \mathbb{R}^3 \right\}.
     \end{align*}  In \eqref{Q+_weak}, we replace $v, v'$ with $v_1, v_1'$ respectively, and choose $\phi(v_1)=D_j(v_1)$. Then we observe that
     \begin{align*}
     \begin{split}
           &\int_{\mathbb{R}^3} D_j (v_1) Q^{+}(f,f)(t,v_1) \; dv_1 \\
        &= \frac{1}{\beta^2} \int_{\mathbb{R}^3}\int_{\mathbb{R}^3} f(t,v_1)f(t,v_*)|v_1-v_*|^{-1} \int_{ B\left(\frac{\beta}{2}|v_1-v_*| \right)} D_j(v_1') \; d\sigma' dv_1 dv_* \\
        &\leq \frac{\pi}{\beta} \int_{\mathbb{R}^3}\int_{\mathbb{R}^3} f(t,v_1)f(t,v_*) \chi \; dv_1 dv_*  \leq 2\pi\|f\|_{1,0}^2
     \end{split}
     \end{align*} by \eqref{d_j_lessim}. Since the bound is uniform in $j$, we let $j \rightarrow \infty$ and use Lemma \ref{lim_int_Q+} to complete the proof.
\end{proof}

We rewrite \eqref{homo_IB} as 
 \begin{align} \label{boltzmann equation}
        \partial_t f(t,v_1) + Lf(t,v_1) f(t,v_1) = Q^{+}(f,f) (t,v_1).
    \end{align}

\begin{lemma} \label{int_f_E_lemma}
Assume that \eqref{initial contidion} holds and $f_0 \in L_{s}^{\infty}$ for $s>2$. Then there exist positive constant $C_1(\|f_0\|_{\infty,s})$ and $C_2$ such that
\begin{align*} 
    \int_{v_* \in E^{v'-v}_P} f(t,v_*) \;dE_{v_*} \leq C_1(\|f_0\|_{\infty,s})\min{\left\{
    1+t,\; e^{C_2 \left(\frac{1}{\alpha^2}-1 \right)t} \right\}}
\end{align*} for any $v, v' \in \mathbb{R}^3$ and $t >0$.
\end{lemma}
\begin{proof}
    From \eqref{boltzmann equation} and Lemma \ref{Q+_E_estimate}, we have
\begin{align*}
        \partial_t  \int_{v_* \in E^{v'-v}_P} f(t,v_*) \;dE_{v_*} \leq  \int_{v_* \in E^{v'-v}_P} Q^{+}(f,f)(t,v_*) \; dE_{v_*} \leq 2\pi\|f_0\|_{1,0}^2
    \end{align*} for fixed $v,v' \in \mathbb{R}^3$. Note that we have 
    \begin{align} \label{plane_int_0}
    \begin{split}
        &\int_{v_* \in E^{v'-v}_P} f_0(v_*) \;dE_{v_*}  \leq \|f_0 \|_{\infty,s} \int_{\mathbb{R}^2}(1+|v_*|)^{-s} dv_* \\
          &\leq 2\pi\|f_0 \|_{\infty,s} \int_0^{\infty} r(1+r)^{-s} \; dr \\
          &\leq 2\pi\|f_0 \|_{\infty,s} \left( \int_0^1 1 \;dr + \int_1^{\infty} (1+r)^{-s+1}\;dr \right)  \leq 4 \pi\|f_0 \|_{\infty,s}
    \end{split}
    \end{align}
   for $s>2$. Thus we have
  \begin{align} \label{depend_time}
         \int_{v_* \in E^{v'-v}_P} f(t,v_*) \;dE_{v_*} \leq  2\pi\|f_0\|_{1,0}^2 t + 4\pi\|f_0 \|_{\infty,s}.
    \end{align}

    Next, we apply \eqref{Lf>e} and Lemma \ref{Q+_E_estimate} to \eqref{boltzmann equation}, and obtain
    \begin{align} \label{int_f_E_inequ}
    \begin{split}
             &\partial_t  \int_{v_* \in E^{v'-v}_P} f(t,v_*) \;dE_{v_*} + C_1 e^{-C_2 \left(\frac{1}{\alpha^2}-1 \right)t} \int_{v_* \in E^{v'-v}_P} f(t,v_*) \;dE_{v_*} \\
             &\leq 2\pi\|f_0\|_{1,0}^2
    \end{split} 
    \end{align} for some positive constants $C_1$ and $C_2$. Then we obtain
    \begin{align} \label{int_f_E}
        \int_{v_* \in E^{v'-v}_P} f(t,v_*) \;dE_{v_*} \leq C_3(\|f_0\|_{\infty,s}) e^{C_2 \left(\frac{1}{\alpha^2}-1 \right)t} 
    \end{align} for some positive constant $C_3(\|f_0\|_{\infty,s})$, by using Lemma \ref{sup_f} and \eqref{plane_int_0}. Lastly, we combine \eqref{depend_time} and \eqref{int_f_E}.
\end{proof}

In the proof of Lemma \ref{f_0_estimate}, we divide the region of \(v_1\) into subsets represented by \(O_1\), \(O_2\), and \(O_3\) for $v \neq 0$. Specifically, we define:

\begin{align} \label{O2,O3_def}
\begin{split}
&O_1 = \left\{v_1 \in \mathbb{R}^3: \; |v_1| < \frac{|v|}{2}\right\}, \\
&O_2 = \left\{ v_1 \in \mathbb{R}^3 : |v_1| \geq \frac{|v|}{2} \right\} \cap \left\{ v_1 \in \mathbb{R}^3 : |v-v_1| < \frac{1}{2}|v|^{-p} \right\}, \\
&O_3 = \left\{ v_1 \in \mathbb{R}^3 : |v_1| \geq \frac{|v|}{2} \right\} \cap \left\{ v_1 \in \mathbb{R}^3 : |v-v_1| \geq \frac{1}{2}|v|^{-p} \right\},
\end{split}
\end{align} for some \(p \in \mathbb{R}\). To estimate the region of \(O_3\) in \eqref{O3}, we need to restrict the range of \(p\) to \(p+1 \leq 2\). For the region of \(O_2\) in \eqref{O2}, \(2p \geq 1\). This allows us to choose \(p\) such that \(\frac{1}{2} \leq p \leq 1\). In the proof, we set \(p = \frac{1}{2}\).

\begin{lemma} \label{f_0_estimate}
Assume that \eqref{initial contidion} holds, and $f \in L_{0}^{\infty}$. Then there exists a positive constants $C(\|f\|_{\infty,0})$ such that
\begin{align} \label{int_t0}
      \int_{\mathbb{R}^3}  f(t,v_1)  |v-v_1|^{-1} \; dv_1\leq 
      C(\|f\|_{\infty,0}) (1+|v|)^{-1}
  \end{align} for $v \in \mathbb{R}^3$ and $t>0$.
\end{lemma}
\begin{proof}
   Assume $v \neq 0$ holds. We define the sets $O_1$, $O_2$, and $O_3$ as in \eqref{O2,O3_def}, where $p = \frac{1}{2}$. Since $|v_1-v| > |v|/2$ for $v_1 \in O_1$, we have 
    \begin{align*}
        \int_{O_1} f(t,v_1)|v-v_1|^{-1} \;dv_1 \leq 2 |v|^{-1} \|f_0\|_{1,0}.
    \end{align*} For $v_1 \in O_2$, we have
    \begin{align} \label{O2}
    \begin{split}
           &\int_{O_2} f(t,v_1)|v-v_1|^{-1} \; dv_1 \\
        &\leq  \|f\|_{0, \infty} \int_{\{v_1\;:\;|v-v_1| < \frac{1}{2}|v|^{-\frac{1}{2}}\}} |v-v_1|^{-1} \; dv_1 =\frac{1}{2}\pi\|f\|_{0, \infty}|v|^{-1}.
    \end{split}
    \end{align} For $v_1 \in O_3$, we have 
    \begin{align} \label{O3}
    \begin{split}
            \int_{O_3} f(t,v_1)|v-v_1|^{-1} \; dv_1 &= \int_{O_3} f(t,v_1)(1+|v_1|^2) 
        \frac{1}{1+|v_1|^2}|v-v_1|^{-1}\; dv_1 \\ &\leq \; \|f_0 \|_{1,2} \frac{8\sqrt{v}}{4+|v|^2}  \leq \; 2\|f_0 \|_{1,2} |v|^{-1}.
    \end{split}
    \end{align} This implies that 
    \begin{align} \label{O123}
         &\int_{\mathbb{R}^3} f(t,v_1)|v-v_1|^{-1} \; dv_1 \sum_{i=1,2,3}\int_{O_i} f(t,v_1)|v-v_1|^{-1} \; dv_1
        \leq C(\|f\|_{\infty,0}) |v|^{-1}
    \end{align} for some positive constant $C(\|f\|_{\infty,0})$.
    On the other hand, we have
    \begin{align} \label{v_include_0}
    \begin{split}
        &\int_{\mathbb{R}^3} f(t, v_1) |v-v_1|^{-1} \; dv_1 \\ &\leq  
           \int_{\mathbb{R}^3} f(t,v_1) \left( \chi_{\{v_1 : |v-v_1| \geq 1\}}(v_1)+ \chi_{\{v_1 :|v-v_1| \leq 1 \}}(v_1) \right) |v-v_1|^{-1} \; dv_1 \\
              &\leq \|f_0\|_{1,0} + 2\pi\|f\|_{\infty,0}
    \end{split}
    \end{align} for $v \in \mathbb{R}^3$. Therefore we obtain \eqref{int_t0} from \eqref{O123} and \eqref{v_include_0}.
\end{proof}

In Lemma \ref{f_esti_lemma}, we divide the range of \(v\) to account for whether it includes zero or not. We cannot apply Lemma \ref{Lf>|v|} if \(|v|\) does not have a uniform lower bound. Consequently, the time variable \(t\) appears on the left side of \eqref{f_esti_ineq_all_v_min}, which includes the case where \(v = 0\).

\begin{lemma} \label{f_esti_lemma} 
Assume that \eqref{initial contidion} holds and $f_0 \in L_{0}^{\infty}$. Then there exist positive constants $C_1(\|f_0\|_{\infty,0}), C_2(\|f_0\|_{\infty,0})$ and $C_3$ such that
  \begin{align} \label{v neq 0}
      \int_{\mathbb{R}^3}  f(t,v_1)  |v-v_1|^{-1} \; dv_1\leq C_1(\|f_0\|_{\infty,0}) |v|^{-1}
  \end{align} for $v \neq 0 \in \mathbb{R}^3$ and $t>0$, and
  \begin{align} \label{f_esti_ineq_all_v_min}
      \int_{\mathbb{R}^3}  f(t,v_1)  |v-v_1|^{-1} \; dv_1\leq
     C_2(\|f_0\|_{\infty,0})(1+|v|)^{-1}\min{\left\{1+t,\; e^{C_3 \left(\frac{1}{\alpha^2}-1 \right)t} \right\}}
  \end{align} for $v \in \mathbb{R}^3$ and $t>0$.
 
\end{lemma} 
\begin{proof}
Assume $v\neq 0$ holds. In the case of $|v_1| < |v|/2$, we have
    \begin{align} \label{|v_1|<|v|/2}
        \int_{\mathbb{R}^3} f(t,v_1) \chi_{\left\{v_1 : |v_1| <\frac{|v|}{2}\right\}}(v_1) |v-v_1|^{-1} \; dv_1
        &\leq 2|v|^{-1}\|f_0\|_{1,0}.
    \end{align}
    For $|v_1| \geq |v|/2>0$, we have
    \begin{align}\label{ineq_int_f_|v_1|>|v|/2}
            \begin{split}
        &\partial_t \int_{\mathbb{R}^3} f(t,v_1)\chi_{\left\{v_1:|v_1| \geq \frac{|v|}{2}\right\}}(v_1)|v-v_1|^{-1} \; dv_1 \\ &\quad\quad +
        \int_{\mathbb{R}^3} Lf(t,v_1)f(t,v_1)\chi_{\left\{v_1:|v_1| \geq \frac{|v|}{2}\right\}}(v_1) |v-v_1|^{-1} \; dv_1 \\
        &= \int_{\mathbb{R}^3} Q^{+}(f,f)(t,v_1)\chi_{\left\{v_1:|v_1| \geq \frac{|v|}{2}\right\}}(v_1) |v-v_1|^{-1} \; dv_1
    \end{split} 
    \end{align} from \eqref{boltzmann equation}. By applying \eqref{Lf_1} and Lemma \ref{Q+_estimate} to \eqref{ineq_int_f_|v_1|>|v|/2}, we obtain
    \begin{align} \label{v neq 0 before gronwall}
    \begin{split}
        &\partial_t \int_{\mathbb{R}^3} f(t,v_1)\chi_{\left\{v_1:|v_1| \geq \frac{|v|}{2}\right\}}(v_1)|v-v_1|^{-1} \; dv_1
        \\ &\quad\quad +\pi\frac{|v|}{2} \int_{\mathbb{R}^3} f(t,v_1) \chi_{\left\{v_1:|v_1| \geq \frac{|v|}{2}\right\}}(v_1) |v-v_1|^{-1} \; dv_1
        \leq 4\pi\|f_0\|_{1,0}^2.
    \end{split}
    \end{align} From Lemma \ref{f_0_estimate}, we have
    \begin{align} \label{initial_v>1_|v-v|}
    \begin{split}
        &\int_{\mathbb{R}^3} f_0(v_1) |v-v_1|^{-1} \; dv_1 \leq C_1(\|f_0\|_{\infty,0}) (1+|v|)^{-1}
    \end{split}
    \end{align} for some positive constant $C_1(\|f_0\|_{\infty,0})$, at $t=0$. Therefore, by applying Lemma \ref{sup_f} and \eqref{initial_v>1_|v-v|} to \eqref{v neq 0 before gronwall}, we obtain
    \begin{align} \label{|v_1|>|v|/2}
    \begin{split}
        \int_{\mathbb{R}^3} f(t,v_1)\chi_{\left\{v_1:|v_1| \geq \frac{|v|}{2}\right\}}(v_1) |v-v_1|^{-1} \; dv_1 \leq |v|^{-1}\left( 8 \|f_0\|_{1,0}^2 + C_1(\|f_0\|_{\infty,0})\right). 
        \end{split}
    \end{align} We deduce \eqref{v neq 0} from \eqref{|v_1|<|v|/2} and \eqref{|v_1|>|v|/2}.
  
  Now, we consider all range of $v \in \mathbb{R}^3$. From \eqref{boltzmann equation} and Lemma \ref{Q+_estimate}, we have
    \begin{align} \label{f_v-v1_int}
         \partial_t \int_{\mathbb{R}^3} f(t,v_1)|v-v_1|^{-1} \; dv_1 
        \leq 4\pi\|f_0\|_{1,0}^2.
    \end{align} 
Integrating \eqref{f_v-v1_int}, we obtain 
    \begin{align} \label{f_esti_ineq_all_v}
      \int_{\mathbb{R}^3}  f(t,v_1)  |v-v_1|^{-1} \; dv_1\leq
      4\pi\|f_0\|_{1,0}^2 t +  \|f_0\|_{1,0} + 2\pi\|f_0\|_{\infty,0}
  \end{align} from \eqref{v_include_0}. Similarly, if we use \eqref{Lf>e}, then we have 
    \begin{align} \label{int_f_v-v1_ineq}
    \begin{split}
        &\partial_t \int_{\mathbb{R}^3} f(t,v_1)|v-v_1|^{-1}dv_1 \\&+C_2 e^{-C_3 \left(\frac{1}{\alpha^2}-1 \right)t} \int_{\mathbb{R}^3} f(t,v_1)|v-v_1|^{-1}dv\leq 4\pi\|f_0\|_{1,0}^2
    \end{split}
    \end{align} for some positive constants $C_2, C_3$, and obtain
    \begin{align} \label{f_esti_ineq_alpha}
        \int_{\mathbb{R}^3} f(t,v_1)|v-v_1|^{-1} dv_1 \leq C_4(\|f_0\|_{\infty,0}) e^{C_3 \left(\frac{1}{\alpha^2}-1 \right)t}
    \end{align} for some positive constant $C_4(\|f_0\|_{\infty,0})$, by Lemma \ref{sup_f} and \eqref{initial_v>1_|v-v|}.
For $0 \leq |v| \leq 1$,
since $1 \leq 2(1+|v|)^{-1}$, there is a positive constant $C_5(\|f_0\|_{\infty,0})$ such that
\begin{align*}
      \int_{\mathbb{R}^3}  f(t,v_1)  |v-v_1|^{-1} \; dv_1\leq
     C_5(\|f_0\|_{\infty,0})(1+|v|)^{-1}\min{\left\{1+t,\; e^{C_3 \left(\frac{1}{\alpha^2}-1 \right)t} \right\}}
\end{align*} from \eqref{f_esti_ineq_all_v} and \eqref{f_esti_ineq_alpha}. 
For $|v| > 1$, since $|v|^{-1} < 2(1+|v|)^{-1}$, there is a positive constant $C_6(\|f_0\|_{\infty,0})$ such that 
\begin{align*}
      \int_{\mathbb{R}^3}  f(t,v_1)  |v-v_1|^{-1} \; dv_1\leq
    C_6(\|f_0\|_{\infty,0})(1+|v|)^{-1}
\end{align*} from \eqref{v neq 0}. From the two inequalities above, we derive \eqref{f_esti_ineq_all_v_min}.
\end{proof}  

\begin{proof}[Proof of Theorem \ref{L_infty_thm}]
 We recall a Carleman type representation \eqref{Q+_Cal}
\begin{align} \label{Calreman_recall}
\begin{split}
        &Q^{+}(f,f)(t,v_1)  \\
        &\leq
        \frac{1}{\beta^2}\int_{\mathbb{R}^3}\int_{\mathbb{R}^3} \delta(v'-v_1) f(t,v) |v'-v|^{-1}\int_{v_* \in E^{v'-v}_P} f(t,v_*) \;dE_{v_*} dv dv'.
\end{split}
\end{align}
By applying Lemma \ref{int_f_E_lemma} and \eqref{f_esti_ineq_all_v_min} to \eqref{Calreman_recall}, we have
\begin{align}\label{Q_thm}
    Q^{+}(f,f)(t,v_1) \leq C_1(\|f_0\|_{\infty,s})(1+|v_1|)^{-1}\min{\left\{(1+t)^2,\; e^{C_2 \left(\frac{1}{\alpha^2}-1 \right)t} \right\}}
\end{align} for some positive constants $C_1(\|f_0\|_{\infty,s})$ and $C_2$.
From \eqref{boltzmann equation} and \eqref{Q_thm}, we have
\begin{align*}
    \partial_t f (t,v_1)\leq C_1(\|f_0\|_{\infty,s})(1+|v_1|)^{-1}(1+t)^2,
\end{align*} and by integrating with respect to time, then obtain 
\begin{align} \label{1}
    f (t,v_1)\leq C_3(\|f_0\|_{\infty,s})(1+|v_1|)^{-1}(1+t)^3
\end{align} for some positive constant $C_3(\|f_0\|_{\infty,s})$.

Now, by applying \eqref{Lf>e} and \eqref{Q_thm} to \eqref{boltzmann equation}, we have
\begin{align} \label{alpha_B_t}
\begin{split}
    &\partial_t f(t,v_1) +C_4 e^{-C_5 \left(\frac{1}{\alpha^2}-1 \right)t}  (1+|v_1|)f(t,v_1) \\
    &\leq  C_1(\|f_0\|_{\infty,s}) (1+|v_1|)^{-1}e^{C_2 \left(\frac{1}{\alpha^2}-1 \right)t}
\end{split}
\end{align} for some positive constant $C_4$ and $C_5$. Then, by multiplying both sides of \eqref{alpha_B_t} by $(1+|v_1|)^2$ and applying Lemma \ref{sup_f}, then we obtain 
\begin{align} \label{2}
      \esssup_{v \in \mathbb{R}^3 } f(t,v)(1+|v|)^2 \leq C_6(\|f_0\|_{\infty,s})e^{C_7\left(\frac{1}{\alpha^2}-1\right)t}
\end{align} for some positive constants $C_6(\|f_0\|_{\infty,s})$ and $C_7$, under $f_0 \in L_{s}^{\infty}$ for $s>2$.

Let us assume $v_1 \neq 0$. By Lemma \ref{int_f_E_lemma} and\eqref{v neq 0}, we have
\begin{align} \label{Q_thm_epsilon}
    Q^{+}(f,f)(t,v_1) \leq C_8(\|f_0\|_{\infty,s})|v_1|^{-1}(1+t)
\end{align} for some positive constant $C_8(\|f_0\|_{\infty,s})$. By \eqref{Lf_1} and \eqref{Q_thm_epsilon}, 
\begin{align*}
     \partial_t f(t,v_1)  + \pi \frac{|v_1|}{2} f(t,v_1) \leq C_8(\|f_0\|_{\infty,s})|v_1|^{-1}(1+t).
\end{align*}  Then we obtain 
\begin{align} \label{3}
    \esssup_{v \neq 0} f(t, v) |v|^2 \leq C_9(\|f_0\|_{\infty,s}) (1+t)
\end{align} for some positive constant $C_9(\|f_0\|_{\infty,s})$, by using Lemma \ref{sup_f}. We derive \eqref{thm_v include 0} from \eqref{1} and \eqref{2}, and derive \eqref{thm_v except 0} from \eqref{2} and \eqref{3}.
\end{proof}

\begin{proof} [Proof of Corollary \ref{Tc_Cor}]
   We assume $T_c < +\infty$, where $T_c$ is in \eqref{Tc}, and let $r>0$. Since
   \begin{align*}
      \mathcal{E}(T_c)= \int_{\mathbb{R}^3} f(T_c,v)|v|^2 \;dv  \geq r^2 \int_{\{|v|>r\}} f(T_c,v) \;dv
   \end{align*} and $\mathcal{E}(T_c)=0$, we obtain 
   \begin{align} \label{v>r}
       \int_{\{|v|>r\}} f(T_c,v) \;dv=0.
   \end{align} Using \eqref{thm_v include 0} and \eqref{v>r}, we have
   \begin{align*}
       1=\int_{\mathbb{R}^3} f(T_c,v)\; dv
       &=\int_{\{|v|<r\}} f(T_c,v) \;dv \\
       &\leq  \frac{2\pi^2}{3}r^3 \|f(T_c,.)\|_{\infty} \leq  C (\|f_0\|_{\infty,s})r^3(1+T_c)^3
   \end{align*} for some positive constant $C(\|f_0\|_{\infty,s})$. It is a contradiction if we choose \\
   $r< \left (C(\|f_0\|_{\infty,s})(1+T_c)^3 \right)^{-1}$, and we get $T_c = +\infty$.
   \end{proof}

This section on the $L^{\infty}$ estimates of the solution of the inelastic Boltzmann equation concludes here. Using the estimates, we will prove the formation of the Maxwellian upper-bound in the next section.

\section{Maxwellian upper bound for the inelastic Boltzmann equation }
\label{sec.Maxupper}
In this section, we will extend Theorem \ref{L_infty_thm} and obtain Maxwellian upper bounds for the inelastic Boltzmann equation for each time. In \cite{GPV2009}, Gamba, Panferov, and Villani proved Lemma \ref{comparison} for the homogeneous Boltzmann equation. Since \eqref{weak_colision_operator} also holds for the inelastic collision operator, we can prove Lemma \ref{comparison} similarly. Lemma \ref{comparison} is a significant component in the proof of Theorem \ref{thm_upper}.

\begin{lemma} \label{comparison} Let $f, u$ satisfy
    \begin{align}\label{comparison_condi_0}
        f(t,v) \geq 0  \text{\quad and \quad} u(0,v) \leq 0 
    \end{align} for $t>0$ and almost all $v \in \mathbb{R}^3$. Assume that
    \begin{align} \label{comparison_condi_1}
        \partial_t u- Q(f,u)\leq 0 \text{\quad on \quad} |v|  > R 
    \end{align} and
    \begin{align}\label{comparison_condi_2}
        u \leq 0  \text{\quad on \quad} 0 \leq |v| \leq R 
    \end{align} for some $0 < R < \infty$ and $t>0$.
    Then, we have $u \leq 0$ for $t>0$ and almost all $v \in \mathbb{R}^3$.
\end{lemma}
\begin{proof} We define $u_{+}(t,v)=\max  \big\{ u(t,v), 0 \big\} $. The function sign $z$ is defined as 1 for $z>0$, $-1$ for $z<0$ and arbitrary fixed value in $[-1,1]$ for $z=0$. We have
\begin{align*}
      \int_{\mathbb{R}^3}( u_{+}(t,v) - u_{+}(0,v))  \; dv =\int_{0}^{t} \int_{\mathbb{R}^3} \partial_t u \times \frac{1}{2}( \text{sign}\; u(s,v) + 1 )\; dv ds.
\end{align*} 
By \eqref{weak_colision_operator}, \eqref{comparison_condi_1}, \eqref{comparison_condi_2} and $u_{+}(0,v)=0$ from \eqref{comparison_condi_0}, we have
    \begin{align*}
        &\int_{\mathbb{R}^3} u_{+}(t,v) \; dv \leq \int_{0}^{t} \int_{\mathbb{R}^3} Q(f,u) \frac{1}{2}\big( \text{sign}\; u(s,v) + 1 \big) \; dv ds \\
        &= \frac{1}{2}\int_{0}^{t} \int_{\mathbb{R}^3} \int_{\mathbb{R}^3} \int_{\mathbb{S}_{+}^2} |(v-v_*)\cdot n |f(s,v_*) u(s,v) \\&\qquad\quad  \times\big(\text{sign}\; u(s,v')-\text{sign}\; u(s,v) \big) \; dn dv_* dvds \leq 0.
    \end{align*} 
    Because $u(s,v)\big(\text{sign}\; u(s,v')-\text{sign}\; u(s,v) \big) \leq 0$ for every $0<s<t$, $v \in \mathbb{R}^3$, we obtain $u_{+}(t,v) \leq 0$ for almost all $v \in \mathbb{R}^3$.
\end{proof}

To apply Lemma \ref{comparison} to Theorem \ref{thm_upper}, we need to find \(R > 0\) in \eqref{comparison_condi_1}, where \(u(t,v) = f(t,v) - \overline{M}(t,v)\) for some Maxwellian-type time-dependent distribution \(\overline{M}(t,v)\). Later, we will determine the specific form of \(\overline{M}(t,v)\). In the proof of Lemma \ref{Q(M)<0} below, we use the inequality \(M(v)M(v_*) \leq M(v')M(v'_*)\) for a global Maxwellian $M$ which holds by \eqref{loss_energy}.

\begin{lemma} \label{Q(M)<0}
    Let $M(v_1) = e^{-a|v_1|^2+b}$ for $a>0$.
    We assume that
    \begin{align} \label{C}
        \int_{\mathbb{R}^3} f(t,v) e^{2a|v|^2}\; dv \leq C
    \end{align} for some constant $C>0$ and $t>0$.
    Then, there exists $0<R<\infty$ such that 
    \begin{align*}
         Q(f,M)(t,v_1)\leq 0
    \end{align*} for $|v_1|>R$. Here $R$ depends on $C$ and $a$, but not on $b$.
\end{lemma}
\begin{proof} 
We plug $M(v)$ into $g(t,v)$ of \eqref{def_Q+} to obtain
\begin{align*}
    &Q^{+}(f,M)(t,v_1) =  \int_{\mathbb{R}^3}\int_{\mathbb{R}^3}\int_{ \mathbb{S}_{+}^2}
    &|(v-v_*)\cdot n| f(t,v_*)M(v)\delta(v'-v_1)\: dn  dv_*dv.
\end{align*} We also recall in \eqref{def_theta} that
\begin{align*}
     \cos \theta = \left( \frac{v_*-v}{|v_*-v|}, n \right), 
     \quad \theta \in \left[0,\frac{\pi}{2} \right],
\end{align*} and in \eqref{def_beta} that
\begin{align*}
    \beta = \frac{1 + \alpha}{2}, \quad \frac{1}{2} < \beta \leq 1.
\end{align*}
From Figure \ref{fig_v'}, we obtain
\begin{align} \label{cos_beta}
      \cos \theta = \frac{|v_*-v_*'|}{\beta|v-v_*|},
      \text{\quad and \quad} 
      \sin \theta = \frac{|v_*-P|}{|v-v_*|}
\end{align} for
\begin{align} \label{recall_P}
     P &= \frac{1}{\beta}v'-\left(\frac{1}{\beta}-1\right)v,
\end{align} and 
\begin{align} \label{v'v*}
    |v'-v_*| \geq (1-\beta) |v-v_*|.
\end{align}
By the fact that $M(v)M(v_*) \leq M(v')M(v_*')$ and \eqref{coordinate_n,sigma}, we have
\begin{align} \label{split_A}
\begin{split}
     &Q^{+}(f,M)(t,v_1)  \\
     &=  \frac{1}{4} \int_{\mathbb{R}^3}\int_{\mathbb{R}^3}\int_{ \mathbb{S}^2}
    |v-v_*| \frac{f(t,v_*)}{M(v_*)}M(v'_*)M(v')\delta(v'-v_1)\: d\sigma  dv_*dv.
\end{split}
\end{align} In terms of $(v',v_*')$ in \eqref{v'_sigma}, we define
\begin{align}\label{def_A}
    &A= A(v,v_*,\sigma;\beta):=\left\{v,v_* \in \mathbb{R}^3,\sigma \in  \mathbb{S}^2 : \frac{1}{\beta}-1 \leq  \frac{1}{2}\frac{|v'-v_*|}{|v_*-v_*'|}        \right\}
\end{align} and
\begin{align}
    &A^{c}= A^{c}(v,v_*,\sigma;\beta):=\left\{v,v_* \in \mathbb{R}^3,\sigma \in  \mathbb{S}^2 : \frac{1}{\beta}-1 \geq  \frac{1}{2}\frac{|v'-v_*|}{|v_*-v_*'|}        \right\}. \label{def_Ac}
\end{align} Moreover, we split $Q^{+}(f,M)(t,v_1)$ into $Q_{A}^{+}(f,M)(t,v_1)$ and $Q_{A^c}^{+}(f,M)(t,v_1)$, where
\begin{align*}
      &Q_{A}^{+}(f,M)(t,v_1) \\
      &:=  \frac{1}{4} \int_{\mathbb{R}^3}\int_{\mathbb{R}^3}\int_{ \mathbb{S}^2}
    |v-v_*| \frac{f(t,v_*)}{M(v_*)}M(v'_*)M(v')\delta(v'-v_1)\mathbf{1}_{A}\: d\sigma  dv_*dv
\end{align*} and
\begin{align*}
     &Q_{A^{c}}^{+}(f,M)(t,v_1) \\
      &:=  \frac{1}{4} \int_{\mathbb{R}^3}\int_{\mathbb{R}^3}\int_{ \mathbb{S}^2}
    |v-v_*| \frac{f(t,v_*)}{M(v_*)}M(v'_*)M(v')\delta(v'-v_1)\mathbf{1}_{A^{c}}\: d\sigma  dv_*dv.
\end{align*}
We now estimate $Q_{A}^{+}(f,M)(t,v_1)$ for $\frac{1}{2} <\beta < 1$. By \eqref{cos_beta}, we have
\begin{align}\label{Q+fM_1}
\begin{split}
     &Q_{A}^{+}(f,M)(t,v_1)  \\
     &\leq  \frac{1}{4}\left(\frac{1}{\beta}\right)^{\epsilon_1}\int_{\mathbb{R}^3}\int_{\mathbb{R}^3}\int_{ \mathbb{S}^2}
     |v-v_*|^{1-\epsilon_1+\epsilon_2}
     \frac{|v_*-v_*'|^{\epsilon_1} }{|v_*-P|^{\epsilon_2}}  \frac{\sin^{\epsilon_2}\theta}{\cos^{\epsilon_1}\theta}  \mathbf{1}_{A} \\
     &\quad\times \frac{f(t,v_*)}{M(v_*)}M(v_*')M(v')\delta(v'-v_1)
    \; d\sigma dv_* dv 
\end{split}
\end{align} for $0<\epsilon_1, \epsilon_2<1$. 
For $v,v_*,\sigma \in A$, we have
\begin{align} \label{in A}
\begin{split}
    |v_*-P | &\geq |v'-v_*| - |v'-P|\\
    &= |v'-v_*| - \left( \frac{1}{\beta}-1 \right)|v_*-v'_*| \geq \frac{1}{2}|v'-v_*|.
\end{split}
\end{align}
Applying \eqref{in A} to \eqref{Q+fM_1}, and \eqref{v'v*} to \eqref{apply  v'v*}, respectively, we have
\begin{align} 
     &\text{(RHS) of\;} \eqref{Q+fM_1} \\
     &\leq  \frac{2^{\epsilon_2}}{4}\left(\frac{1}{\beta}\right)^{\epsilon_1}\int_{\mathbb{R}^3}\int_{\mathbb{R}^3}\int_{ \mathbb{S}^2}  \frac{|v-v_*|^{1-\epsilon_1+\epsilon_2}}{|v'-v_*|^{\epsilon_2}}
     |v_*-v_*'|^{\epsilon_1}\frac{\sin^{\epsilon_2}\theta}{\cos^{\epsilon_1}\theta} \notag  \\
     &\quad\times \frac{f(t,v_*)}{M(v_*)}M(v_*')M(v')\delta(v'-v_1)
    \; d\sigma dv_* dv   \label{apply  v'v*}\\
    &\leq \frac{2^{\epsilon_2}}{4}\left(\frac{1}{\beta}\right)^{\epsilon_1}
    \left( \frac{1}{1-\beta}\right)^{\epsilon_2}\int_{\mathbb{R}^3}\int_{\mathbb{R}^3}\int_{ \mathbb{S}^2}  |v-v_*|^{1-\epsilon_1}
     |v_*-v_*'|^{\epsilon_1}\frac{\sin^{\epsilon_2}\theta}{\cos^{\epsilon_1}\theta} \notag  \\
     &\quad\times \frac{f(t,v_*)}{M(v_*)}M(v_*')M(v')\delta(v'-v_1)
    \; d\sigma dv_* dv.  \label{Q+fM_2}
\end{align} By using the fact that $|v-v_*| \leq |v'-v_*|+|v-v'| = 
|v'-v_*|+|v_*-v_*'|, \;|v_*'|M(v_*')\lesssim 1$, and $|v_*|e^{-a|v_*|^2} \lesssim 1$, we have
\begin{align}\label{Q+fM_3}
\begin{split}
     \eqref{Q+fM_2}
    &\lesssim \frac{2^{\epsilon_2}}{4}\left(\frac{1}{\beta}\right)^{\epsilon_1}
    \left( \frac{1}{1-\beta}\right)^{\epsilon_2}\int_{\mathbb{R}^3} f(t,v_*)e^{2a|v_*|^2} \\
    &\quad \times \int_{ \mathbb{S}^2}  \frac{\sin^{\epsilon_2}\theta}{\cos^{\epsilon_1}\theta} 
 \int_{\mathbb{R}^3}   (1+|v'|)^{1-\epsilon_1} M(v')\delta(v'-v_1)
    \; dv d\sigma dv_*.   
\end{split}
\end{align} From \eqref{v'_sigma}, we can calculate the Jacobian determinant for the map between $v$ and $v'$ for fixed $\sigma$ as 
\begin{align} \label{jaco_v'v}
\begin{split}
    \left| \frac{dv'}{dv} \right| &= \left(1- \frac{\beta}{2}\right)^3  \left(1 + \frac{\beta}{2-\beta} \frac{v-v_*}{|v-v_*| } \cdot \sigma \right) \\
     &=\left(1- \frac{\beta}{2}\right)^3 \left(1- \frac{\beta}{2-\beta} + \frac{2\beta}{2-\beta}\sin^2 \theta \right)>0
\end{split}
\end{align} since \eqref{theta-2theta}. Applying \eqref{jaco_v'v} to \eqref{Q+fM_3}, and $\frac{1}{2} < \beta < 1$, we have
\begin{align}
     &\text{(RHS) of\;} \eqref{Q+fM_3}\\
     &\lesssim \frac{2^{\epsilon_2}}{4}\left(\frac{1}{\beta}\right)^{\epsilon_1}
    \left( \frac{1}{1-\beta}\right)^{\epsilon_2}\left(1- \frac{\beta}{2}\right)^{-3} 
    \int_{\mathbb{R}^3} f(t,v_*)e^{2a|v_*|^2} dv_* \notag\\
    &\quad \times \int_{\mathbb{R}^3}   (1+|v'|)^{1-\epsilon_1} M(v')\delta(v'-v_1) dv' \notag \\
    &\quad \times \int_0^{\frac{\pi}{2}} \sin 2\theta 
    \frac{\sin^{\epsilon_2}\theta}{\cos^{\epsilon_1}\theta} \left(  \frac{2\beta}{2-\beta}\sin^2 \theta    \right)^{-1} d\theta \notag \\
    &\lesssim   \left( \frac{1}{1-\beta}\right)^{\epsilon_2}M(v_1) (1+|v_1|)^{1-\epsilon_1}\int_{\mathbb{R}^3} f(t,v_*)e^{2a|v_*|^2} dv_* \notag\\
    &\quad \times \left(  \int_0^{\frac{\pi}{4}} \sin^{\epsilon_2-1} \theta \;d\theta +
     \int_{\frac{\pi}{4}}^{\frac{\pi}{2}} \cos^{1-\epsilon_1}\theta \;d\theta \right)
     \label{Q+fM_4}
\end{align} for $0<\epsilon_1, \epsilon_2<1$. Taking $\epsilon_2 = 1-\beta $ and $0<\epsilon_1<1$ in \eqref{Q+fM_4} and using \eqref{C}, we have
\begin{align} \label{esti_Q_A}
     Q_{A}^{+}(f,M)(t,v_1) \lesssim M(v_1)(1+|v_1|)^{1-\epsilon_1}
\end{align} since $(1-\beta)^{-(1-\beta)}\lesssim 1$.

Next, we estimate $Q_{A^{c}}^{+}(f,M)(t,v_1)$ for $\frac{1}{2} <\beta < 1$. By
\eqref{cos_beta}, we have 
\begin{align}
       &Q_{A^{c}}^{+}(f,M)(t,v_1) \notag \\
        &\leq  \frac{1}{4}\left(\frac{1}{\beta}\right)^{\epsilon_1}\int_{\mathbb{R}^3}\int_{\mathbb{R}^3}\int_{ \mathbb{S}^2}
     |v-v_*|^{1-\epsilon_1}|v_*-v_*'|^{\epsilon_1}
      \frac{1}{\cos^{\epsilon_1}\theta}  \mathbf{1}_{A^{c}} \notag\\
     &\quad\times \frac{f(t,v_*)}{M(v_*)}M(v_*')M(v')\delta(v'-v_1)
    \; d\sigma dv_* dv \label{Q+fm_1}
\end{align} for $0<\epsilon_1<1$. By using the fact that $|v-v_*| \leq |v'-v_*|+|v_*-v_*'|, \;|v_*'|M(v_*')\lesssim 1$, and $|v_*|e^{-a|v_*|^2} \lesssim 1$, we have
\begin{align} \label{Q+fm_2}
\begin{split}
     \text{(RHS) of\;} \eqref{Q+fm_1} &\lesssim  \frac{1}{4}\left(\frac{1}{\beta}\right)^{\epsilon_1}\int_{\mathbb{R}^3} f(t,v_*)e^{2a|v_*|^2} \\
    &\quad \times \int_{ \mathbb{S}^2} \frac{1}{\cos^{\epsilon_1}\theta}
     \int_{\mathbb{R}^3}   (1+|v'|)^{1-\epsilon_1} M(v')\delta(v'-v_1)  \mathbf{1}_{A^{c}} 
    \; dv d\sigma dv_*.   
\end{split}
\end{align} From \eqref{def_Ac}, we have
\begin{align*} 
\begin{split}
    \frac{1}{\beta}-1 &\geq \frac{1}{2} \frac{|v'-v_*|}{|v_*-v'_*|} 
    \geq \frac{1}{2}\left( \frac{|v-v_*|}{|v_*-v_*'|}-\frac{|v-v'|}{|v_*-v_*'|} \right)\\
    &=\frac{1}{2}\left( \frac{1}{\beta\cos\theta}  -1 \right)>0
\end{split} 
\end{align*} for $v,v_*,\sigma \notin A$,  and it yields that
\begin{align}\label{1/beta-1}
    \cos \theta \geq \frac{1}{2-\beta}.
\end{align}
From \eqref{jaco_v'v}, and obtain
\begin{align} \label{dv'/dv_lower}
\begin{split}
     \left| \frac{dv'}{dv} \right| 
     &\geq \left(1- \frac{\beta}{2}\right)^3 \left(1- \frac{\beta}{2-\beta} \right) = 2  \left(1- \frac{\beta}{2}\right)^3\frac{1-\beta}{2-\beta}>0.
\end{split} 
\end{align} Applying \eqref{dv'/dv_lower} to \eqref{Q+fm_2}, and using \eqref{1/beta-1}, we have
\begin{align}\label{Q+fm_3}
\begin{split}
      \text{(RHS) of\;} \eqref{Q+fm_2} &\lesssim \frac{1}{4}\left(\frac{1}{\beta}\right)^{\epsilon_1}\frac{2-\beta}{1-\beta}\left(1- \frac{\beta}{2}\right)^{-3}    \int_{\mathbb{R}^3} f(t,v_*)e^{2a|v_*|^2} dv_* \\
        &\quad \times \int_{\mathbb{R}^3}   (1+|v'|)^{1-\epsilon_1} M(v')\delta(v'-v_1) dv' \\
        &\quad \times 
        \int_0^{\frac{\pi}{2}} \sin 2\theta \frac{1}{\cos^{\epsilon_1}\theta} \mathbf{1}_{\left\{ \cos\theta \geq \frac{1}{2-\beta} \right\}} d \theta. 
\end{split}
\end{align}
Because
\begin{align*}
        \int_0^{\frac{\pi}{2}} \sin 2\theta \frac{1}{\cos^{\epsilon_1}\theta} \mathbf{1}_{\left\{ \cos\theta \geq \frac{1}{2-\beta} \right\}} d \theta
        &\leq 2\int_0^{\frac{\pi}{2}} \sin\theta \cos^{1-\epsilon_1}\theta \mathbf{1}_{\left\{ \cos\theta \geq \frac{1}{2-\beta} \right\}} d \theta \\
        &\leq 2\int_{\frac{1}{2-\beta}}^1  t^{1-\epsilon_1} dt \leq 2\frac{1-\beta}{2-\beta}
\end{align*} for $0<\epsilon_1<1$, from \eqref{Q+fm_3}, we have
\begin{align} \label{esti_Q_Ac}
       Q_{A^{c}}^{+}(f,M)(t,v_1) \lesssim M(v_1)(1+|v_1|)^{1-\epsilon_1}
\end{align} since $\frac{1}{2} < \beta < 1$. Combining \eqref{esti_Q_A} with \eqref{esti_Q_Ac}, there exists some constant $C_1>0$ such that
\begin{align*}
    Q^{+}(f,M)(f,M) &=  Q_{A}^{+}(f,M)(t,v_1)+ Q_{A^{c}}^{+}(f,M)(t,v_1)\\
    &\leq C_1 M(v_1)(1+|v_1|)^{1-\epsilon_1}
\end{align*} for $0< \epsilon_1 < 1$. From Lemma \ref{Lf>|v|}, we have 
   \begin{align*}
       Q^{-}(f,M) (t,v_1) = M(v_1) Lf(t,v_1) \geq \pi M(v_1)|v_1|.
   \end{align*} Therefore we have
   \begin{align*} 
        Q(f,M)(t,v_1) \leq M(v_1) \left(C_1(1+|v_1|)^{1-\epsilon_1} -\pi|v_1| \right).
   \end{align*} Then we choose $R>0$ sufficiently large such that $Q(f,M) \leq 0$, and it concludes the proof for $0<\beta<1$. In Lemma 5 of \cite{GPV2009}, the authors proved an analogue of this Lemma in the case of elastic collisions, when $\beta =1$.
\end{proof}

Now we are ready to prove the Maxwellian upper bound of the solution. We apply Lemma \ref{Q(M)<0} to \eqref{comparison_condi_1} and use the inequality 
\begin{align} \label{apply_to_maxwellian}
    f(t,v) \leq C_1(\|f_0\|_{\infty,s})\min{\left\{(1+t)^3, e^{C_2\left(\frac{1}{\alpha^2}-1\right)t}\right\}}
\end{align} for some positive constants $C_1(\|f_0\|_{\infty,s}), \;C_2$, from Theorem \ref{L_infty_thm} to \eqref{comparison_condi_2}. Additionally, \eqref{comparison_condi_0} arises from the assumption of the initial bound.

\begin{proof}[Proof of Theorem \ref{thm_upper}]
There exists a positive constant $K_1$, which depends on $c_0$ such that
\begin{align} \label{C_1_thm}
    \int_{\mathbb{R}^3} f_0(v)e^{\frac{a_0}{2}|v|^2} dv
    \leq K_1.
\end{align} since $f_0(v) \leq M_0(v)$.
By Proposition 3.2-(ii) in \cite{SCM2006}, and $T_c = +\infty$ in Corollary \ref{Tc_Cor} or Remark \ref{remark_Tc}, and \eqref{C_1_thm}, there exist $K_2, a_1>0$, which depend on $a_0$ and $K_1$ such that
    \begin{align} \label{C_2_thm}
        \sup_{t\in[0,\infty)} \int_{\mathbb{R}^3} f(t,v) e^{a_1|v|^{2}} dv \leq K_2.
    \end{align} We take $a=\min\{a_0, \frac{a_1}{2}\}$ and $b \geq c_0$ and set
\begin{align*}
    u(t,v)= f(t,v_1)-\overline{M}(t,v_1),
    \text{\; where \;} \overline{M}(t,v_1)=(1+t)^3 e^{-a|v_1|^2+b}.
\end{align*} Then we have
  \begin{align*}
       u(0,v) = f(0,v_1)-\overline{M}(0,v_1) \le  f_0(v_1)-M_0(v_1) \leq 0,
   \end{align*} since $a<a_0$ and $c_0 \leq b$. By \eqref{C_2_thm} and Lemma \ref{Q(M)<0}, there exists a large $R>0$, which depends on $a_1, K_2$ such that
   \begin{align*}
       \partial_t u - Q(f,u) (v_1) 
       = -\partial_t\overline{M} + Q(f,\overline{M}) (t,v_1)
       \leq 0
   \end{align*} for $|v_1| \geq R$. From \eqref{apply_to_maxwellian}, we have $f(t,v_1) \leq C_1 (t+1)^3$, for some constant $C_1$ depending on $c_0$. Taking $b=aR^2+ \log C_1 (\geq c_0)$, it holds that
   \begin{align*}
      \overline{M}(t,v_1)= C_1 (1+t)^3 e^{-a|v_1|^2+aR^2} \geq f(t,v_1)
   \end{align*} for $|v_1|<R$. Therefore we apply Lemma \ref{comparison} and obtain 
   \begin{align} \label{upper_max_poly}
       f(t,v_1) \leq \overline{M}(t,v_1)=(1+t)^3 e^{-a|v_1|^2+b}.
    \end{align} 
    Moreover, we replace
    \begin{align*}
         \overline{M}(t,v_1)=(1+t)^3 e^{-a|v_1|^2+b}
    \end{align*} with 
    \begin{align*}
          \widehat{M}(t,v_1)=e^{C_2\left(\frac{1}{\alpha^2}-1\right)t} e^{-a|v_1|^2+b},
    \end{align*} where the positive constant $C_2$ is used in \eqref{apply_to_maxwellian}, and depends on $c_0$. Then, by following the same argument as above, we obtain
       \begin{align} \label{upper_max_exp}
       f(t,v_1) \leq     \widehat{M}(t,v_1)=e^{C_2\left(\frac{1}{\alpha^2}-1\right)t} e^{-a|v_1|^2+b}
    \end{align} instead of \eqref{upper_max_poly}. From \eqref{upper_max_poly} and \eqref{upper_max_exp}, we conclude
    \begin{align*}
          f(t,v_1) \leq \min{\left\{(1+t)^3, e^{C_2\left(\frac{1}{\alpha^2}-1\right)t} \right\}} e^{-a|v_1|^2+b}
    \end{align*} for almost every $v_1 \in \mathbb{R}^3$ and every $t > 0$.
\end{proof}

\section*{Acknowledgment}
The first and the third authors are supported by the National Research Foundation of Korea(NRF) grants RS-2023-00212304 and RS-2023-00219980. The second author is supported by the National Research Foundation of Korea (NRF) grants RS-2023-00210484, RS-2023-00219980, 2022R1G1A1009044 and 2021R1A6A1A100\\42944, and also by Posco Holdings Inc. and Samsung Electronics Co., Ltd. (IO23041\\2-05903-01). \\ 

\noindent \textbf{Data Availability} Not applicable\\

\noindent \textbf{Conflict of interest} The authors declare that they have no conflict of interest.

\bibliographystyle{amsplain3links}
\nocite{*}
\bibliography{reference.bib}
\end{document}